\documentclass[11pt]{amsart}
\usepackage[utf8]{inputenc}
\usepackage{amscd,amssymb,amsthm,amsmath,amssymb,textcomp,supertabular,rotating,longtable,enumerate,rotating,mathrsfs,mathtools,multirow}
\usepackage[matrix,arrow,curve]{xy}
\usepackage{pdflscape}

\usepackage[shortlabels]{enumitem}

\usepackage[margin=2cm]{geometry}

\newtheorem*{theorem*}{Theorem}

\newtheorem*{proposition*}{Proposition}
\newtheorem{lemma}[equation]{Lemma}
\newtheorem{corollary}[equation]{Corollary}
\newtheorem*{corollary*}{Corollary}

\newtheorem*{problem*}{Problem}

\newtheorem*{question*}{Question}

\newtheorem*{construction*}{Construction}
\newtheorem*{maintheorem*}{Main Theorem}

\theoremstyle{definition}
\newtheorem{example}[equation]{Example}
\newtheorem*{example*}{Example}

\newtheorem*{definition*}{Definition}
\newtheorem{Family}{Family}

\theoremstyle{remark}
\newtheorem{remark}[equation]{Remark}
\newtheorem*{remark*}{Remark}

\makeatletter\@addtoreset{equation}{section} \makeatother

\newcommand{\mumu}{\boldsymbol{\mu}}

\title{One-dimensional components in the K-moduli of smooth Fano 3-folds}

\author{Hamid Abban, Ivan Cheltsov, Elena Denisova, Erroxe Etxabarri-Alberdi, \\ Anne-Sophie Kaloghiros, Dongchen Jiao,
Jesus Martinez-Garcia, Theodoros Papazachariou}

\let\origmaketitle\maketitle
\def\maketitle{
  \begingroup
  \def\uppercasenonmath##1{} 
  \let\MakeUppercase\relax 
  \origmaketitle
  \endgroup
}

\address{\emph{Hamid Abban}
\newline
\textnormal{School of Mathematical Sciences, University of Nottingham, Nottingham NG7 2RD, UK}
\newline
\textnormal{\texttt{hamid.abban@nottingham.ac.uk}}}

\address{\emph{Ivan Cheltsov}
\newline
\textnormal{School of Mathematics, University of Edinburgh, Edinburgh EH9 3FD,   UK}
\newline
\textnormal{\texttt{i.cheltsov@ed.ac.uk}}}

\address{\emph{Elena Denisova}
\newline
\textnormal{School of Mathematics, University of Edinburgh, Edinburgh EH9 3FD,   UK}
\newline
\textnormal{\texttt{e.denisova@sms.ed.ac.uk}}}

\address{\emph{Erroxe Etxabarri-Alberdi}
\newline
\textnormal{School of Mathematical Sciences, University of Nottingham, Nottingham NG7 2RD, UK}
\newline
\textnormal{\texttt{erroxe.etxabarrialberdi@nottingham.ac.uk}}}

\address{\emph{Anne-Sophie Kaloghiros}
\newline
\textnormal{Department of Mathematics, Brunel University London, Uxbridge UB8 3PH, UK}
\newline
\textnormal{\texttt{anne-sophie.kaloghiros@brunel.ac.uk}}}

\address{\emph{Dongchen Jiao}
\newline
\textnormal{Department of Mathematics, Brunel University London, Uxbridge UB8 3PH, UK}
\newline
\textnormal{\texttt{dongchen.jiao@brunel.ac.uk}}}

\address{\emph{Jesus Martinez-Garcia}
\newline
\textnormal{Department of Mathematical Sciences, University of Essex, Colchester, CO4 3SQ, UK}
\newline
\textnormal{\texttt{jesus.martinez-garcia@essex.ac.uk}}}

\address{\emph{Theodoros Papazachariou}
\newline
\textnormal{School of Mathematics \& Statistics, University of Glasgow, Glasgow G12 8QQ,  UK}
\newline
\textnormal{\texttt{theodorosstylianos.papazachariou@glasgow.ac.uk}}}

\begin{document}

\begin{abstract} By identifying K-polystable limits in 4 specific deformations families of smooth Fano 3-folds, we complete the classification of one-dimensional components in the K-moduli space of smoothable Fano 3-folds.
\end{abstract}

\maketitle


\section{Introduction}

By recent advances in the theory of K-stability, there is a projective moduli space $M^{\mathrm{Kps}}_{n}$
whose closed points parametrise $n$-dimensional K-polystable smoothable Fano varieties over $\mathbb{C}$, see the survey by Xu \cite{Xu} and the references therein.
Components of $M^{\mathrm{Kps}}_{2}$
have been studied in \cite{OdakaSpottiSun}.
The next step is to  analyse components of $M^{\mathrm{Kps}}_{3}$
together with all K-polystable smoothable Fano 3-folds.

According to the classification by Iskovskikh, Mori and Mukai, there are 105 deformation families of smooth Fano 3-folds.
Among these, 27 families contain no K-polystable smooth 3-folds \cite{Book,Fujita2016}, while the general element of the remaining 78 families is K-polystable \cite{Book,Fujita2021}. Note that in one case (\textnumero 2.26 in the Mori-Mukai notation), there are no K-polystble smooth objects, but the family contains a unique singular K-polystable member \cite[Section 5.10]{Book}. To study $M^{\mathrm{Kps}}_{3}$, it would be sensible to work inductively on the dimension of components appearing in this moduli space. Among the 78 families under study, 20 deformation families contain a unique K-polystable smooth object\footnote{In Mori-Mukai notation, these are families numbered 1.15, 1.16, 1.17, 2.27, 2.29, 2.32, 2.34, 3.15, 3.17,
3.19, 3.20, 3.25, 3.27, 4.3, 4.4, 4.6, 4.7, 5.1, 5.3 and 6.1.}, hence their corresponding K-moduli components are isolated points. Next up is to study the one-dimensional components.

\subsection{One-dimensional components in $M^{\mathrm{Kps}}_{3}$}

The goal of this paper is to describe one-dimensional components of $M^{\mathrm{Kps}}_{3}$.
There are 8 families of smooth Fano 3-folds with one-dimensional moduli,
but only 6 of them have K-polystable elements. For these 6 families, all smooth K-polystable members are known \cite{Book,CheltsovPark,Denisova}, and for 2 families,
the singular K-polystable limits are also known. Let us describe K-polystable smooth Fano 3-folds in these 6 deformation families, and introduce their K-polystable singular limits.

\begin{Family} Divisors of bidegree $(1,2)$ in $\mathbb{P}^2\times\mathbb{P}^2$. (\textnumero 2.24 in Mori-Mukai notation)
\label{example:2-24}

For $\lambda\in\mathbb{P}^1$, let $X_\lambda$ be the 3-fold defined  by
$
\{xu^2+yv^2+zw^2=\lambda(xvw+yuw+zuv)\} \subset \mathbb{P}^2\times \mathbb{P}^2,
$
where $([x:y:z],[u:v:w])$ are coordinates on $\mathbb{P}^2\times \mathbb{P}^2$.
Then $X_\lambda$ is smooth if and only if $\lambda^3\neq 1 $ or $\lambda\neq\infty$. By \cite[Lemma 4.70]{Book} $X_\lambda$ are K-polystable  for all $\lambda\in\mathbb{P}^1$. In particular, if $\lambda^3=1$, then $X_\lambda\cong X_{\infty}$, and this 3-fold has $3$ ordinary double points.
Moreover, every K-polystable smooth Fano 3-fold in this family is isomorphic to $X_\lambda$ for some $\lambda\in\mathbb{P}^1$.
Note that the family contains strictly K-semistable smooth members (see \cite[Section 4.7]{Book} for details).
\end{Family}

\begin{Family} Blowups of $\mathbb{P}^3$ along quartic elliptic curves. (\textnumero 2.25 in Mori-Mukai notation)
\label{example:2-25}

For $\lambda\in\mathbb{P}^1$, consider the curve
$
C_\lambda=\big\{x_0^2+x_1^2+\lambda(x_2^2+x_3^2)=0,\lambda(x_0^2-x_1^2)+x_2^2-x_3^2=0\big\}\subset\mathbb{P}^3,
$
where $[x_0:x_1:x_2:x_3]$ are coordinates on $\mathbb{P}^3$, and let $\pi\colon X_\lambda\to\mathbb{P}^3$ be the blowup along $C_\lambda$.
If $\lambda\not\in\{0,\pm 1,\pm i,\infty\}$, then $C_\lambda$ is a smooth elliptic curve,
and $X_\lambda$ is a smooth K-stable Fano 3-fold \cite[Corollary 4.32]{Book}.
Moreover, every smooth Fano 3-fold in this family is isomorphic to $X_\lambda$ for some $\lambda\in\mathbb{P}^1$.
If $\lambda\in\{0,\pm 1,\pm i,\infty\}$, then $C_\lambda$ is a union of $4$ lines,
and $X_\lambda\cong X_0$ is a toric K-polystable smoothable Fano 3-fold; $X_0$ has four singular points, which are ordinary double points.
\end{Family}

\begin{Family} Blowups of $\mathbb{P}^3$ along rational quartic curves. (\textnumero 2.22 in Mori-Mukai notation)
\label{example:2-22}

Let $Q$ be the smooth quadric surface  $\{x_0x_3=x_1x_2\}\subset\mathbb{P}^3$,
where $[x_0:x_1:x_2:x_3]$ are coordinates on~$\mathbb{P}^3$.
Identify $Q$ with $\mathbb{P}^1\times \mathbb{P}^1$ via the isomorphism given by
$$
\big([u:v],[x:y]\big)\rightarrow\big[xu:xv:yu:yv\big],
$$
where $([u:v],[x:y])$ are coordinates on $\mathbb{P}^1\times \mathbb{P}^1$.
Let $C_\lambda\subset Q$ be the curve defined by
$
\{ux^2(x+\lambda y)=vy^2(y+\lambda x)\},
$
for $\lambda\in\mathbb{P}^1$.
One can check that $C_\lambda$ is smooth if and only if $\lambda\not\in\{\pm 1,\infty\}$, and it is then a rational quartic curve.
If $\pi\colon X_\lambda\to \mathbb{P}^3$ is the blowup along $C_\lambda$, then $X_\lambda$ is a smoothable Fano 3-fold. Moreover, every (smooth) member of family \textnumero 2.22 is isomorphic to $X_\lambda$ for some $\lambda\in\mathbb{P}^1$.
The 3-fold  $X_\lambda$ is K-polystable for $\lambda\not\in\{\pm 1,\pm 3,\infty\}$ by \cite{CheltsovPark}.
On the other hand, $X_{\pm 3}$ is strictly K-semistable,
with K-polystable limit $X_0$.

Since $C_{\pm 1}$ is the union of a twisted cubic and a line,
and $C_\infty$ is a union of a conic and two lines,
$X_{\pm 1}$ admits an isotrivial degeneration to $X_\infty$.
Hence, if $X_{\infty}$ is K-polystable, then $X_{\pm 1}$ is strictly K-semistable.
Note that $X_{\infty}$ has two singular points, which are ordinary double points.
\end{Family}

\begin{Family} Blowups of $\mathbb{P}^3$ along the disjoint union of a twisted cubic and a line. (\textnumero 3.12 in Mori-Mukai notation)
\label{example:3-12}

In the notation of Family~\ref{example:2-22}, identify $\mathbb P^1\times \mathbb P^1$ and $C_\lambda$ with subvarieties of $\mathbb{P}^1\times\mathbb{P}^2$ via the embedding
$$
\big([u:v],[x:y]\big)\mapsto\big([u:v],[x^2:xy:y^2]\big).
$$
Let $\pi\colon X_\lambda\to\mathbb{P}^1\times\mathbb{P}^2$ be the blowup along the curve $C_\lambda$, then $X_\lambda$ is a smoothable Fano 3-fold. Further,
every (smooth) member of family \textnumero 3.12 is isomorphic to $X_\lambda$ for some $\lambda\in\mathbb{P}^1\setminus\{\pm 1,\infty\}$.
Moreover, if $\lambda\not\in\{\pm 3,\pm 1,\infty\}$, then $X_\lambda$ is K-polystable \cite{Denisova}.
The smooth Fano 3-fold $X_{\pm 3}$ is strictly K-semistable, with K-polystable limit $X_0$.
Since the (singular) 3-fold $X_{\pm 1}$ admits an isotrivial degeneration to $X_{\infty}$, if $X_\infty$ K-polystable,  $X_{\pm 1}$ is strictly K-semistable.
Note that $X_{\infty}$ has two singular points, which are ordinary double points.
\end{Family}

\begin{Family} Blowups of $\mathbb{P}^1\times\mathbb{P}^1\times\mathbb{P}^1$ along a curve of degree $(1,1,3)$. (\textnumero 4.13 in Mori-Mukai notation)
\label{example:4-13}

For $\lambda\in\mathbb{P}^1$, let $\pi\colon X_\lambda\to (\mathbb{P}^1)^3$ be the blowup along the curve
$$
C_\lambda=\big\{x_0y_1-x_1y_0=0, x_0^3z_0-x_1^3z_1+\lambda x_0x_1(x_1z_0-x_0z_1)=0\big\},
$$
where $([x_0:x_1],[y_0:y_1],[z_0:z_1])$ are the coordinates on $(\mathbb{P}^1)^3$.
If $\lambda\not\in\{\pm 1,\infty\}$, then $X_\lambda$ is a smooth K-polystable Fano 3-fold. Further,
every (smooth) member of family \textnumero 4.13 is isomorphic to $X_\lambda$ for some $\lambda\in\mathbb{P}^1$.
One can show that $X_{\pm 1}\cong X_{\infty}$, and the 3-fold $X_{\infty}$ has two singular points, which are ordinary double points.
\end{Family}

\begin{Family} Complete intersection of divisors of degree $(1,1,0)$, $(1,0,1)$, and $(0,1,1)$ in $\mathbb{P}^2\times\mathbb{P}^2\times\mathbb{P}^2$. (\textnumero 3.13 in Mori-Mukai notation)
\label{example:3-13}

For $\lambda\in\mathbb{P}^1$, let $X_\lambda \subset (\mathbb{P}^2)^3$ be given by
\begin{equation}
\label{equation:3-13-old}
\big\{x_0y_0+x_1y_1+x_2y_2=0, y_0z_0+y_1z_1+y_2z_2=0, (1+\lambda)x_0z_1+(1-\lambda)x_1z_0-2x_2z_2=0\big\},
\end{equation}
where $([x_0:x_1:x_2],[y_0:y_1:y_2],[z_0:z_1:z_2])$ are coordinates on $(\mathbb{P}^2)^3$.
If $\lambda\not\in\{\pm 1,\infty\}$, then $X_\lambda$ is a smooth K-polystable Fano 3-fold. Further,
every (smooth) member of family \textnumero 3.13 is isomorphic to $X_\lambda$ for some $\lambda\in\mathbb{P}^1$.
For $\lambda\in \{\pm 1, \infty\}$, $X_{\lambda}$ is singular ($X_{\pm 1}$ has one ordinary double point and $X_\infty$ is singular along a curve) and K-unstable.

This family can reparametrised in such a way that $X_\lambda$ degenerates to the toric K-polystable Fano $3$-fold $\{x_0y_1=x_1y_0, y_1z_2=y_2z_1, x_0z_2=x_2z_0\}$ when $\lambda\to\pm 1$. Note that this $3$-fold is singular; it has $3$ ordinary double points.
\end{Family}

Now, we are ready to state our main result:

\begin{maintheorem*}
Let $X$ be one of the 3-folds $X_{\infty}$ described in Families \ref{example:2-22}, \ref{example:3-12}, \ref{example:4-13}.
Then $X$ is K-polystable.
In the notation of Family \ref{example:3-13}, let
\begin{equation}
\label{equation:3-13-non-toric}
X_\infty^\prime=\left\{\aligned
&x_2y_3-x_3y_2=0,\\
&y_2z_3-y_3z_2=0,\\
&x_2z_3-x_3z_2=0,\\
&x_1y_1z_3+x_1y_3z_1+x_3y_1z_1+x_3y_2z_3=0,\\
&x_1y_1z_2+x_1y_2z_1+x_2y_1z_1+x_2y_3z_2=0.
\endaligned
\right.
\end{equation}
Then $X_\infty^\prime$ has a unique singular point, which is an ordinary double point, and is a K-polystable limit of elements of family \textnumero 3.13.
\end{maintheorem*}

\begin{corollary}
\label{corollary:limits}
The Fano 3-fold $X_{\infty}$ in Families \ref{example:2-24}, \ref{example:2-25}, \ref{example:2-22}, \ref{example:3-12}, \ref{example:4-13}
are the only singular K-polystable limits of members of the deformation families \textnumero 2.24, 2.25, 2.22, 3.12 and 4.13.
\end{corollary}

\begin{proof}
We only consider Family \ref{example:2-22}, since the proof is similar for other families.
Denote by $M^{\mathrm{Kps}}_{2.22}$ the one-dimensional component of the K-moduli space $M^{\mathrm{Kps}}_3$
that contains all smooth K-polystable Fano 3-folds in Family \ref{example:2-22} (equivalently, all K-polystable elements of Mori-Mukai family \textnumero 2.22).
Above, we described a parametrisation $\big\{X_{\lambda} ; \lambda \in \mathbb{P}^1\big\}$ that is a $\mathbb{Q}$-Gorenstein family, and such that all smooth members of Family \ref{example:2-22} are fibres of the family $X_{\lambda}$ for $\lambda \in \mathbb{P}^1\setminus\{\pm 3,\pm 1,\infty\}$. Note that $X_{\lambda}\cong X_{-\lambda}$ for  $\lambda\in\mathbb{P}^1$.

Moreover, it follows from the description of the Family \ref{example:2-22} above and the Main Theorem
that all objects $X_\lambda$ in the parametrisation except for the 3-folds $X_{\pm 3}$ and $X_{\pm 1}$  are K-polystable. As mentioned already,
the 3-folds $X_{\pm 3}$ and $X_{\pm 1}$ are K-semistable, and their K-polystable limits are $X_0$ and $X_{\infty}$, respectively. Thus we have a morphism $\mathbb{P}^1\to\mathcal{M}^{\mathrm{Kss}}_{2.22}$, the moduli stack parametrising K-semistable objects in this family, which descends to a morphism $\phi\colon \mathbb{P}^1\to M^{\mathrm{Kps}}_{2.22}$ given by $\lambda\mapsto [X_\lambda]$
such than $\phi(0)=\phi(\pm 3)$, $\phi(\infty)=\phi(\pm 1)$, and $\phi(\lambda)=\phi(-\lambda)$ for $\lambda\in\mathbb{P}^1$.
Since $M^{\mathrm{Kps}}_{2.22}$ is proper and one-dimensional,
we conclude that $\phi$ is surjective, which implies the required assertion.
\end{proof}

\begin{corollary}
\label{corollary:3-13}
Singular K-polystable limits of smooth Fano 3-folds
in the Mori-Mukai family \textnumero 3.13 are the toric Fano 3-fold described in Family \ref{example:3-13}
above and the non-toric Fano 3-fold $X^\prime_{\infty}$ defined in \eqref{equation:3-13-non-toric}.
\end{corollary}

\begin{proof}
Let $M^{\mathrm{Kps}}_{3.13}$ be the one-dimensional component of $M^{\mathrm{Kps}}_3$
that contains K-polystable smooth Fano 3-folds in this deformation family.
It follows from the description in Family \ref{example:3-13} that there exists a $\mathbb{Q}$-Gorenstein family of Fano 3-folds over $\mathbb{P}^1$
such that the fibre $X_\lambda$ over $\lambda\in\mathbb{P}^1$
is the complete intersection in $\mathbb{P}^2\times\mathbb{P}^2\times\mathbb{P}^2$ given by \eqref{equation:3-13-old}.
This family contains all smooth K-polystable 3-folds in the family \textnumero 3.13,
which are fibres over the points in $\mathbb{P}^1\setminus\{\pm 1,\infty\}$.
Note that $X_{\lambda}\cong X_{-\lambda}$ for every $\lambda\in\mathbb{P}^1$.
Arguing as in the proof of Corollary~\ref{corollary:limits},
we see that there is a surjective morphism
$\phi\colon \mathbb{P}^1\to M^{\mathrm{Kps}}_{3.13}$ such that $\phi(\lambda)=[X_\lambda]$ for $\lambda\in\mathbb{P}^1\setminus\{\pm 1,\infty\}$,
and $\phi(\pm 1)$ is the K-polystable toric Fano 3-fold described in Family \ref{example:3-13}.

For $\lambda\ne\infty$, the K-polystable Fano 3-fold corresponding to $\phi(\lambda)$ is either smooth or
has ordinary double points, in particular, $X_\lambda$ has unobstructed $\mathbb{Q}$-Gorenstein deformations.
So, it follows from \cite[Remark 2.4]{KaloghirosPetracci} that
$M^{\mathrm{Kps}}_{3.13}$ is smooth at $\phi(\lambda)$ for $\lambda\ne\infty$.
It follows from Main Theorem that
$[X_{\infty}^\prime]\in M^{\mathrm{Kps}}_{3.13}$,
where $X_{\infty}^\prime$ is the 3-fold \eqref{equation:3-13-non-toric}.
But $[X_{\infty}^\prime]\ne\phi(\lambda)$ for $\lambda\ne\infty$, since $X_{\infty}^\prime\not\cong X_\lambda$ for $\lambda\not\in\{0,\infty\}$,
and $X_{\infty}^\prime$ is not isomorphic to the toric Fano 3-fold  described in Family \ref{example:3-13}.
Thus, we conclude that $\phi(\infty)=[X_{\infty}^\prime]$, so that $M^{\mathrm{Kps}}_{3.13}$ is smooth at $\phi(\lambda)$,
which gives $M^{\mathrm{Kps}}_{3.13}\cong\mathbb{P}^1$.
\end{proof}

\begin{corollary}
All one-dimensional components of $M^{\mathrm{Kps}}_{3}$ are isomorphic to $\mathbb{P}^1$.
\end{corollary}


\subsection{The mirage of GIT}
Explicitly describing the K-moduli for a given deformation family is no easy task at present. In all known cases, the K-moduli space either coincides with some GIT moduli space or is closely related to it by blowing up certain subspaces in the GIT moduli. Out approach to proving the Main Theorem was guided by this phenomenon but with a new hands-on approach. We first write down a parametrisation of the objects and then examine their limits for K-polystability. Hidden in that approach is the hope that the K-polystable limit has the same description as the smooth objects; in other words it lives in the same ambient space with similar defining equations --- that is to say it follows some GIT principle. This works almost perfectly in our studies, and one obtains that the K-moduli space is the same as a suitable natural GIT moduli for Families \ref{example:2-24}--\ref{example:4-13}. Verifying the details of the latter claim could be an interesting research exercise. For instance, the case of Family \ref{example:2-25} is well-studied \cite[\S5]{Papazachariou}. However, as the reader has already observed, the K-polystable limit $X'_\infty$ in Family \ref{example:3-13} is no longer a complete intersection in $\mathbb{P}^2\times\mathbb{P}^2\times\mathbb{P}^2$. Indeed, by a result of Eisenbud--Buchsbaum \cite{BuchsbaumEisenbud} such codimension $3$ subschemes would be given by Pfaffians, rather than being complete intersections. So, we rewrote the parametrisation of the complete intersection of 3 divisors in $\mathbb{P}^2\times\mathbb{P}^2\times\mathbb{P}^2$ that would define the desired (smooth) Fano 3-fold in (redundant) Pfaffian format and studied their Pfaffian limit, which is no longer a complete intersection. This limit is $X'_\infty$ as in the Main Theorem and is K-polystable. Hence, the obvious candidates of GIT compactification (regarded as complete intersection) would not read off the K-moduli in this case. Whether there is another GIT description that coincides with the K-moduli in this case remains to be investigated.

\subsection{Structure of the paper}
In Section~\ref{section:Abban-Zhuang},
we explain the strategy of the proof for the first part of Main Theorem, and provide details in Sections~\ref{section:2-22}, \ref{section:3-12}, \ref{section:4-13}.
In Section~\ref{section:3-13}, we present a very simple geometric construction for the 3-fold \eqref{equation:3-13-non-toric}, and prove the second part of Main Theorem.

\smallskip
\noindent
\textbf{Acknowledgements} We would like to thank Yuchen Liu and Ziquan Zhuang for helpful discussions. This project was initiated at a week-long meeting at the De Morgan House of the London Mathematical Society (LMS), funded by a Focused Research Grant of the Heilbronn Institute.
Several authors are supported by EPSRC grants: Hamid Abban by EP/V048619, Ivan Cheltsov by EP/V054597/1, Anne-Sophie Kaloghiros by EP/V056689/1, and Jesus Martinez-Garcia by EP/V055399/1.

\section{Strategy of the proof}
\label{section:Abban-Zhuang}

In order to prove K-polystability, we use the following powerful result proven by Zhuang in \cite[Corollary~4.14]{Zhuang}: If $G$ is a reductive group acting on the Fano variety $X$ such that $\beta(\mathbf{E})>0$ for all G-invariant prime divisors $\mathbf{E}$ over $X$, then $X$ is K-polystable. By a divisor over $X$, we mean that there exists a birational morphism $\varphi\colon \widehat{X}\to X$ such that $\mathbf{E}$ is a divisor on $\widehat{X}$, and following \cite{Fujita2019}, $\beta(\mathbf{E})$ is defined to be
$$
\beta(\mathbf{\mathbf{E}})=A_X(\mathbf{\mathbf{E}})-S_X(\mathbf{\mathbf{E}}).
$$
In it, $A_X(\mathbf{E})=1+\mathrm{ord}_{\mathbf{E}}\big(K_{\widehat{X}}-\varphi^{*}(K_X)\big)$ is the log discrepancy of the divisor $\mathbf{E}$, and
$$
S_X\big(\mathbf{E}\big)=\frac{1}{(-K_X)^n}\int\limits_0^{\tau(\mathbf{E})}\mathrm{vol}\big(\varphi^*(-K_X)-u\mathbf{E}\big)du,
$$
where
$\tau\big(\mathbf{E}\big)=\mathrm{sup}\big\{u\in\mathbb{R}_{\geqslant 0}\ \vert\ \varphi^*\big(-K_X\big)-u\mathbf{E}\ \text{is big}\big\}$,
and $n$ is the dimension of $X$.
By \cite[Proposition~2.1]{KentoKyoto}, we have that
\begin{equation}
\label{equation:KentoKyoto}
S_X\big(\mathbf{E}\big)\leqslant\frac{n}{n+1}\tau\big(\mathbf{E}\big).
\end{equation}

Let $X$ be one of the singular Fano 3-folds $X_{\infty}$ described in Families \ref{example:2-22}, \ref{example:3-12}, \ref{example:4-13}.
Then $X$ has two isolated ordinary double points, $\mathrm{Aut}^0(X)\cong\mathbb{C}^\ast$, and $\mathrm{Aut}(X)$ is a semi-direct product of $\mathbb{C}^\ast$ and a finite group.
Moreover, in each case  $\mathrm{Aut}(X)$ swaps the singular points of $X$.
In Sections~\ref{section:2-22}, \ref{section:3-12}, \ref{section:4-13},
we will present generators of the group $\mathrm{Aut}(X)$, and we will describe basic geometric facts about $X$.
Set $G=\mathrm{Aut}(X)$.
Then to prove that $X$ is K-polystable, it is enough to show that
$\beta(\mathbf{E})>0$ for every $G$-invariant prime divisor $\mathbf{E}$ over $X$.

Now, let $\varphi\colon\widehat{X}\to X$ be a $G$-equivariant birational morphism with $\widehat{X}$ normal,
and let $\mathbf{F}$ be a $G$-invariant prime divisor in the 3-fold $\widehat{X}$, and $Z=\varphi(\mathbf{F})$ its centre on $X$. Since $G$ swaps singular points of $X$, we have the following possibilities:
$Z$ is a smooth point of $X$,
$Z$ is a $G$-invariant irreducible curve,
$Z$ is a $G$-invariant irreducible surface.

To proceed, it is more convenient to replace $X$ with a suitable $G$-equivariant small resolution.
A priori, such a resolution may not exist, but in all cases considered here, it does, yielding a $G$-equivariant commutative diagram
$$
\xymatrix{
&\widetilde{X}\ar[dr]\ar@{-->}[rr]&&\overline{X}\ar[dl]&\\
&&X&&}
$$
where $\widetilde{X}\to X$ and $\overline{X}\to X$ are small resolutions of singularities of $X$,
and $\widetilde{X}\dasharrow\overline{X}$ is a composition of two Atiyah flops.
Let $Y$ be one of the 3-folds $\widetilde{X}$ or $\overline{X}$,
let $\eta\colon Y\to X$ be the corresponding small $G$-equivariant birational morphism,
and let $Z_Y$ be the centre of the divisor $\mathbf{F}$ on the 3-fold $Y$.
Then $-K_Y\sim\eta^*(-K_X)$, which implies that
$A_X(\mathbf{F})=A_Y(\mathbf{F})$ and $S_X(\mathbf{F})=S_Y(\mathbf{F})$,
where we set
$$
S_Y\big(D\big)=\frac{1}{(-K_Y)^3}\int\limits_0^\infty\mathrm{vol}\big(-K_Y-uD\big)du
$$
for every (not necessarily prime) divisor $D$ over $Y$.

\begin{remark}
\label{remark:Kento-divisorial}
Let $S_1,\ldots,S_r$ be effective divisors on $Y$ such that $S_Y(S_i)<1$ for every $i$.
If~every $G$-invariant prime divisor in $Y$ is linearly equivalent to $\sum\limits_{i=1}^{r}n_iS_i$ for some non-negative integers $n_1,\ldots,n_r$,
then $\beta(S)>0$ for every $G$-invariant prime divisor $S$ in $Y$.
Using \eqref{equation:KentoKyoto}
we can weaken the condition
``every $G$-invariant prime divisor in $Y$ is linearly equivalent to $\sum\limits_{i=1}^{r}n_iS_i$ for some non-negative integers $n_1,\ldots,n_r$''
as follows: for every \mbox{$G$-invariant} prime divisor $D\subset Y$ such that
$-K_Y\sim_{\mathbb{Q}}\frac{4}{3}D+\Delta$ for some effective $\mathbb{Q}$-divisor $\Delta$ on the 3-fold $Y$,
there are non-negative integers $n_1,\ldots,n_r$ such that~$D\sim\sum\limits_{i=1}^{r}n_iS_i$.
Furthermore, using \cite[Proposition~3.2]{Fujita2019alpha}, we can weaken the latter condition slightly as follows:
for every $G$-invariant prime divisor $D\subset Y$ such that
$-K_Y\sim_{\mathbb{Q}}\lambda D+\Delta$
for some rational number $\lambda>\frac{4}{3}$ and some effective $\mathbb{Q}$-divisor $\Delta$ on the 3-fold $Y$,
there are non-negative integers $n_1,\ldots,n_r$ such that $D\sim\sum\limits_{i=1}^{r}n_iS_i$.
\end{remark}

Now, fix a point $P\in Z_Y$ and set
$$
\delta_P\big(Y\big)=\inf_{\substack{\mathbf{E}/Y\\ P\in C_Y(\mathbf{E})}}\frac{A_{Y}(\mathbf{E})}{S_Y(\mathbf{E})},
$$
where the infimum runs over all prime divisors $\mathbf{E}$ over $Y$ whose centre on $Y$ contains $P$.
If $\beta(\mathbf{F})\leqslant 0$ for a divisor $\mathbf{F}$ whose centre contains $P$, then $\delta_P(Y)\leqslant 1$.
Quite often, we can use the inductive argument of Abban and Zhuang \cite{AbbanZhuang}, and its formulation in certain scenarios in \cite{Book}, to show that $\delta_P(Y)>1$.
To do this in the cases we deal with in Families \ref{example:2-22}, \ref{example:3-12}, \ref{example:4-13}, let $\mathscr{C}$ be a smooth irreducible curve in $Y$ that contains $P$,
and let $\mathscr{S}$ be a smooth irreducible surface in $Y$ that contains $\mathscr{C}$.
They provide an admissible flag $P\in\mathscr{C}\subset\mathscr{S}$.
To apply \cite{AbbanZhuang,Book}, set
$$
\tau=\tau\big(\mathscr{S}\big)=\mathrm{sup}\big\{u\in\mathbb{R}_{\geqslant 0}\ \text{the divisor}\ \vert\ -K_Y-u\mathscr{S}\ \text{is big}\big\}.
$$
Next, for every $u\in[0,\tau]$, it is required to find the Zariski decomposition
$$
-K_Y-u\mathscr{S}\sim_{\mathbb{R}} P(u)+N(u),
$$
where $P(u)$ is the positive part of the decomposition, and $N(u)$ is the negative part.
A priori, the Zariski decomposition may not exist on $Y$ for every $u\in[0,\tau]$, but in cases dealt with here, it exists either for $Y=\widetilde{X}$ or for $Y=\overline{X}$.
Hence, we may assume that the required Zariski decomposition exists on $Y$ for every $u\in[0,\tau]$.
For~$u\in[0,\tau]$, set $d(u)=\mathrm{ord}_{\mathscr{C}}(N(u)\vert_{\mathscr{S}})$ and write
$$
N(u)\big\vert_{\mathscr{S}}=N^\prime(u)+d(u)\mathscr{C}
$$
where $N^\prime(u)$ is an effective divisor on $\mathscr{S}$ such that $\mathscr{C}\not\subset\mathrm{Supp}(N^\prime(u))$.
For $u\in[0,\tau]$, set
$$
t(u)=\mathrm{sup}\Big\{v\in\mathbb{R}_{\geqslant 0}\ \big\vert\ \text{the divisor  $P(u)\big\vert_{\mathscr{S}}-v\mathscr{C}$ is pseudo-effective}\Big\}.
$$
Then, for every $v\in[0,t(u)]$, let $P(u,v)$ be the positive part of the Zariski decomposition of~the $\mathbb{R}$-divisor $P(u)\vert_{\mathscr{S}}-v\mathscr{C}$,
and let $N(u,v)$ be its~negative part. Set
$$
S\big(W^S_{\bullet,\bullet};\mathscr{C}\big)=\frac{3}{(-K_X)^3}\int\limits_0^{\tau}d(u)\big(P(u,0)\big)^2du+\frac{3}{(-K_X)^3}\int\limits_0^\tau\int\limits_0^{t(u)}\big(P(u,v)\big)^2dvdu,
$$
which is well defined since the support of $N(u)$ does not contain $\mathscr{S}$ for every $u\in[0,\tau]$.
If $\mathscr{C}=Z_Y$, it follows from \cite{AbbanZhuang,Book} that
\begin{equation}
\label{equation:Kento-curve}
\frac{A_X(\mathbf{F})}{S_X(\mathbf{F})}=\frac{A_Y(\mathbf{F})}{S_Y(\mathbf{F})}\geqslant\min\Bigg\{\frac{1}{S_Y(\mathscr{S})},\frac{1}{S\big(W^\mathscr{S}_{\bullet,\bullet};\mathscr{C}\big)}\Bigg\}.
\end{equation}
Hence, if $\mathscr{C}=Z_Y$, $S_Y(\mathscr{S})<1$ and $S(W^\mathscr{S}_{\bullet,\bullet};\mathscr{C})<1$, then $\beta(\mathbf{F})>0$.
Using this approach, we can show that $\beta(\mathbf{F})>0$ if $Z$ is a $G$-invariant irreducible curve.

\begin{remark}(\cite{AbbanZhuang,Book})
\label{remark:Kento-curve-strict}
In fact, if $\mathscr{C}=Z_Y$, $S_Y(\mathscr{S})<1$ and $S(W^\mathscr{S}_{\bullet,\bullet};\mathscr{C})\leqslant 1$, then  $\beta(\mathbf{F})>0$.
\end{remark}

Now, we observe that $\mathscr{C}\not\subset\mathrm{Supp}(N(u,v))$, and set
$$
F_P\big(W_{\bullet,\bullet,\bullet}^{\mathscr{S},\mathscr{C}}\big)=\frac{6}{(-K_X)^3} \int\limits_0^\tau\int\limits_0^{t(u)}\big(P(u,v)\cdot \mathscr{C}\big)\cdot \mathrm{ord}_P\big(N^\prime(u)\big|_\mathscr{C}+N(u,v)\big|_\mathscr{C}\big)dvdu
$$
and
$$
S\big(W_{\bullet, \bullet,\bullet}^{\mathscr{S},\mathscr{C}};P\big)=\frac{3}{(-K_X)^3}\int\limits_0^\tau\int\limits_0^{t(u)}\big(P(u,v)\cdot \mathscr{C}\big)^2dvdu+F_P\big(W_{\bullet,\bullet,\bullet}^{\mathscr{S},\mathscr{C}}\big).
$$
Then it follows from \cite{AbbanZhuang,Book} that
\begin{equation}
\label{equation:Kento-point}
\frac{A_{Y}(\mathbf{F})}{S_Y(\mathbf{F})}\geqslant
\delta_P\big(Y\big)\geqslant\min\Bigg\{\frac{1}{S_Y(\mathscr{S})},\frac{1}{S\big(W^\mathscr{S}_{\bullet,\bullet};\mathscr{C}\big)},\frac{1}{S\big(W_{\bullet, \bullet,\bullet}^{\mathscr{S},\mathscr{C}};P\big)}\Bigg\}.
\end{equation}
Thus, if $S_Y(\mathscr{S})<1$, $S(W^\mathscr{S}_{\bullet,\bullet};\mathscr{C})<1$ and $S(W_{\bullet, \bullet,\bullet}^{\mathscr{S},\mathscr{C}};P)<1$,
then $\delta_P(Y)>1$ and $\beta(\mathbf{F})>0$.

\begin{remark}[{\cite{AbbanZhuang,Book}}]
\label{remark:Kento-point-strict}
In fact, if $P=Z_Y$, $S_X(\mathscr{S})<1$, $S(W^\mathscr{S}_{\bullet,\bullet};\mathscr{C})\leqslant 1$ and $S(W_{\bullet, \bullet,\bullet}^{\mathscr{S},\mathscr{C}};P)\leqslant 1$,
then we also have $\beta(\mathbf{F})>0$.
\end{remark}

Using this approach, we will show in Sections~\ref{section:2-22}, \ref{section:3-12}, \ref{section:4-13}
that $X_{\infty}$ (in the notation of Families~\ref{example:2-22}, \ref{example:3-12} and  \ref{example:4-13}) is K-polystable. We use similar techniques to treat the case of  $X'_\infty$ in Family \ref{example:3-13} in Section  \ref{section:3-13} .

\section{K-polystability of the Fano 3-fold $X_{\infty}$ in Family \ref{example:2-22}}
\label{section:2-22}

Let $X=X_{\infty}$, where $X_{\infty}$ is described in Family \ref{example:2-22}.
Let $C_{\infty}=C+L_1+L_2$, where $C=\{x_0+x_3=0,x_0x_3=x_1x_2\}$, $L_1=\{x_0=0,x_1=0\}$ and $L_2=\{x_2=0,x_3=0\}$, and let $\gamma\colon V\to \mathbb{P}^3$ be the blowup of the lines $L_1$ and $L_2$,
let $\phi\colon\widetilde{X}\to V$ be the blowup of the proper transform of the conic $C$,
and $\varphi\colon W\to \mathbb P^3$ be the blowup of the conic $C$, and let
$\delta\colon \overline{X}\to W$ be the blowup of the proper transform of the lines $L_1$ and $L_2$.
Then we have the following $G$-equivariant commutative diagram:
$$
\xymatrix{
\widetilde{X}\ar[rr]\ar[d]_{\phi} && X\ar[d]_{\pi} && \overline{X}\ar[ll]\ar[d]^{\delta}\\
V \ar[rr]_{\gamma} && \mathbb{P}^3 && W\ar[ll]^\varphi}
$$
where $\widetilde{X}\to X$ and $\overline{X}\to X$ are $G$-equivariant small resolutions of $X$.
Recall from Section~\ref{section:Abban-Zhuang} that $G=\mathrm{Aut}(X)$,
and either $Y=\widetilde{X}$ or $Y=\overline{X}$.

Recall that $Q=\{x_0x_3=x_1x_2\}\subset\mathbb{P}^3$, and let $E$, $F_1$, $F_2$ be the $\pi$-exceptional surfaces such that $\pi(E)=C$, $\pi(F_1)=L_1$,~\mbox{$\pi(F_2)=L_2$}.
Let $H_C=\{x_0+x_3=0\}$, $H_{C^\prime}=\{x_0=x_3\}$, and denote by $H$ a general plane in $\mathbb{P}^3$,
and denote by $\widetilde{E}$, $\widetilde{F}_1$, $\widetilde{F}_2$, $\widetilde{Q}$, $\widetilde{H}_C$, $\widetilde{H}_{C^\prime}$, $\widetilde{H}$ the proper transforms on $\widetilde{X}$
of the surfaces $E$, $F_1$, $F_2$, $Q$, $H_C$, $H_{C^\prime}$, $H$, respectively.
Then $\widetilde{Q}\sim 2\widetilde{H}-\widetilde{E}-\widetilde{F}_1-\widetilde{F}_2$ and $\widetilde{H}_{C}\sim \widetilde{H}-\widetilde{E}$.
This gives
$$
-K_{\widetilde{X}}\sim 4\widetilde{H}-\widetilde{E}-\widetilde{F}_1-\widetilde{F}_2\sim2\widetilde{Q}+\widetilde{E}+\widetilde{F}_1+\widetilde{F}_2\sim\widetilde{Q}+2\widetilde{H}_{C}+2\widetilde{E}.
$$
Note that $(-K_{\widetilde{X}})^3=(-K_X)^3=30$.
The divisors $\widetilde{H}$, $\widetilde{E}$, $\widetilde{F}_1$, $\widetilde{F}_2$ generate the group $\mathrm{Pic}(\widetilde{X})$.
We have
$\widetilde{H}^3=1$, $\widetilde{H}\cdot\widetilde{F}_1^2=\widetilde{H}\cdot\widetilde{F}_2^2=\widetilde{F}_1\cdot\widetilde{E}^2=\widetilde{F}_2\cdot\widetilde{E}^2=-1$,
$\widetilde{H}\cdot\widetilde{E}^2=\widetilde{F}_1^3=\widetilde{F}_2^3=-2$, $\widetilde{E}^3=-4$,
and all remaining triple intersections are zero.

Similarly, let $\overline{E}$, $\overline{F}_1$, $\overline{F}_2$, $\overline{Q}$, $\overline{H}_C$, $\overline{H}_{C^\prime}$, $\overline{H}$
be the proper transforms on $\overline{X}$ of the surfaces $E$, $F_1$, $F_2$, $Q$, $H_C$, $H_{C^\prime}$, $H$,
respectively. Then
$$
-K_{\overline{X}}\sim 4\overline{H}-\overline{E}- \overline{F}_1-\overline{F}_2\sim
2\overline{Q}+\overline{E}+\overline{F}_1+\overline{F}_2\sim \overline{Q}+2\overline{H}_C+2\overline{E}.
$$
The divisors $\overline{H}$, $\overline{E}$, $\overline{F}_1$, $\overline{F}_2$ generate the group $\mathrm{Pic}(\overline{X})$ and their intersections can be computed as follows:
$\overline{H}^3=1$, $\overline{F}_1^3=\overline{F}_2^3=\overline{F}_1^2\cdot\overline{H}=\overline{F}_2^2\cdot\overline{H}=\overline{F}_1^2\cdot\overline{E}=\overline{F}_2^2\cdot\overline{E}=-1$,
$\overline{E}^2\cdot\overline{H}=-2$, $\overline{E}^3=-6$,
and all remaining triple intersections are zero.

\paragraph*{\bf Description of the automorphism group}
Let $\tau$ be the involution of $\mathbb{P}^3$ given by
$$
[x_0:x_1:x_2:x_3]\mapsto[x_3:x_2:x_1:x_0],
$$
and $\Gamma$ the subgroup of $\mathrm{Aut}(\mathbb{P}^3)$ consisting of automorphisms
$$
[x_0:x_1:x_2:x_3]\mapsto\left[x_0:\lambda x_1:\frac{x_2}{\lambda}: x_3\right]
$$
for $\lambda\in\mathbb{C}^\ast$. Then $\Gamma\cong\mathbb{C}^\ast$,
the curve $C_{\infty}$ is $\langle\tau,\Gamma\rangle$-invariant,
and the $\langle\tau,\Gamma\rangle$-action lifts to~$X$.
Hence, we can identify $\langle\tau,\Gamma\rangle$ with a subgroup in $G=\mathrm{Aut}(X)$.
It is not difficult to verify that $G=\langle \tau,\Gamma\rangle\cong\mathbb{C}^\ast\rtimes\mumu_2$.

\paragraph*{\bf Description of the $G$-invariant loci}

Set $O=[1:0:0:1]$ and $O^\prime=[1:0:0:-1]$.

\begin{lemma}
\label{lemma:2-22-points}
The only $G$-fixed points in $\mathbb{P}^3$ are $O$ and $O^\prime$.
\end{lemma}

\begin{proof}
Left to the reader.
\end{proof}

Let $l=\{x_0=0,x_3=0\}$, $l^\prime=\{x_1=0,x_2=0\}$, $C_r=\{x_1x_2=rx_0x_3,x_0+x_3=0\}$,
$C^\prime_r=\{x_1x_2=rx_0x_3,x_0=x_3\}$ for $r\in\mathbb{C}^\ast$.
Then $l=H_C\cap H_{C^\prime}$ is the line that passes through the points $C\cap L_1$ and $C\cap L_2$,
and $l^\prime$ is the line that passes through  $O$ and $O^\prime$.
Note that $C_r$ is an irreducible conic in the plane $H_C$,
and $C_r^\prime$ is an irreducible conic in the plane $H_C^\prime$.
All these curves   are $G$-invariant, and $C=C_1$. Set $C^\prime=C_1^\prime$.

\begin{lemma}
\label{lemma:2-22-curves}
The curves $l$, $l^\prime$, $C_r$, $C_r^\prime$ are the only $G$-invariant irreducible curves in $\mathbb{P}^3$.
\end{lemma}

\begin{proof}
Let $\mathcal{C}$ be a $G$-invariant irreducible curve in $\mathbb{P}^3$.
If $\mathcal{C}$ is pointwise fixed by $\Gamma$, then $\mathcal C=l^\prime$.
We may assume that $\mathcal{C}\ne l^\prime$.
Then $\Gamma$ acts on $\mathcal{C}$ effectively, which implies that $\mathcal{C}$ is rational.
Then $\tau$ must fix a point $P\in\mathcal{C}$, which is not fixed by $\Gamma$,
which implies that $\mathcal{C}=\overline{\mathrm{Orb}_{\Gamma}(P)}$.
On the other~hand, the $\tau$-fixed points are
$[b:a:a:b]$ and $[b:a:-a:-b]$ for $[a:b]\in\mathbb{P}^1$,
which implies the required assertion.
\end{proof}

Thus, the planes $H_C$ and $H_{C^\prime}$ contain all $G$-invariant irreducible curves in $\mathbb{P}^3$ except $l^\prime$.
To complete the description of $G$-invariant curves in $X$, we have to describe $G$-invariant irreducible curves in $E$,
which is done in the following lemma:

\begin{lemma}
\label{lemma:2-22-curves-E}
The only $G$-invariant irreducible curves in $\widetilde{E}$ are  $\widetilde{E}\cap\widetilde{Q}$ and $\widetilde{E}\cap\widetilde{H}_C$,
and the only $G$-invariant irreducible curves in $\overline{E}$ are $\overline{E}\cap\overline{Q}$ and $\overline{E}\cap\overline{H}_C$.
\end{lemma}

\begin{proof}
Note that $\delta(\overline{E})\cong\mathbb{F}_2$ by \cite[Lemma 2.6]{CheltsovSuess}, and $\delta(\overline{E}\cap\overline{H}_C)$ is the $(-2)$-curve in $\overline{E}$.
Let $\varsigma\colon\mathbb{P}^3\dasharrow\mathbb{P}^4$ be the $G$-equivariant map
$$
[x_0:x_1:x_2:x_3]\mapsto\big[x_0(x_0+x_3):x_1(x_0+x_3):x_2(x_0+x_3):x_3(x_0+x_3):x_0x_3-x_1x_2\big],
$$
and let $Y$ be the closure of its image in $\mathbb{P}^4$.
Then $Y$ is the quadric $\{tw-tx+wx+yz=0\}$,
where $[x:y:z:t:w]$ are coordinates on $\mathbb{P}^4$.
Then $\varsigma$ induces a $G$-equivariant birational map $\sigma\colon\mathbb{P}^3\dasharrow Y$
such that there exists the following $G$-equivariant commutative diagram:
$$
\xymatrix@R=6mm{
&W\ar[dl]_{\varphi}\ar[dr]^\upsilon&\\
\mathbb{P}^3\ar@{-->}[rr]^{\sigma}&& Y}
$$
where $\upsilon$ is the contraction of $\gamma(\overline{H}_C)$ to $[0:0:0:0:1]$.
Set $S_2=\upsilon\circ\gamma(\overline{E})$. Then
$$
S_2=\big\{t+x=0,yz-tx=0\big\}\subset Y,
$$
and $\upsilon$ induces a $G$-equivariant birational morphism $\gamma(\overline{E})\to S_2$
that contracts $\gamma(\overline{E}\cap\overline{H}_C)$.
Moreover, one can check that the only $G$-invariant irreducible curve in the cone $S_2$ is the conic $\{w=t+x=yz-tx=0\}$,
which is the image of the curve $\overline{E}\cap\overline{Q}$. This implies the required assertion.
\end{proof}

In order to use Remark \ref{remark:Kento-divisorial}, we need the following result.

\begin{lemma}
\label{lemma:2-22-Eff}
Let $S$ be a $G$-invariant $G$-irreducible surface in $\widetilde{X}$ such that $S\ne \widetilde{F}_1+\widetilde{F}_2$.
Then $S\sim a\widetilde{Q}+b\widetilde{H}_C+c\widetilde{E}$ for some non-negative integers $a$, $b$, $c$.
\end{lemma}

\begin{proof}
We may assume that $S\ne\widetilde{E}$,
$S\ne\widetilde{H}_{C}$, $S\ne\widetilde{Q}$.
Then $\pi(S)$ is a $G$-invariant surface  of degree $d\geqslant 1$,
and $S\sim d\widetilde{H}-m\widetilde{E}-n(\widetilde{F}_1+\widetilde{F}_2)$
for some non-negative integers $m$ and $n$.
Let $\ell$ be a general ruling of  $\widetilde{Q}\cong\mathbb{P}^1\times\mathbb{P}^1$ such that $\widetilde{F}_1\cdot\ell=\widetilde{F}_2\cdot\ell=1$.
Then $\widetilde{E}_1\cdot\ell=1$ and $0\leqslant S\cdot\ell=d-m-2n$.
Thus, we have $d\geqslant m-2n$. So, we can let $a=n$, $b=d-2n$, $c=d-m-n$.
\end{proof}

We are ready to prove that $X$ is K-polystable using the approach described in Section~\ref{section:Abban-Zhuang}.
Namely, let $\mathbf{F}$ be a $G$-invariant prime divisor over $X$, and let $Z$, $\widetilde{Z}$, $\overline{Z}$ be its centres
on $X$, $\widetilde{X}$ and $\overline{X}$, respectively.
Then it follows from Lemma~\ref{lemma:2-22-points} that
\begin{enumerate}
\item either $Z$ is a $G$-invariant irreducible surface,
\item or $Z$ is a curve described in Lemmas~\ref{lemma:2-22-curves} and \ref{lemma:2-22-curves-E},
\item or $Z$ is a point, and $\pi(Z)$ is one of the points $O$ or $O^\prime$.
\end{enumerate}

We start with the case of a $G$-invariant irreducible surface.
Using Remark~\ref{remark:Kento-divisorial} and Lemma~\ref{lemma:2-22-Eff}, we obtain

\begin{lemma}
\label{lemma:2-22-divisorial}
Let $\mathbf{F}$ be a $G$-invariant prime divisor on $X$. Then $\beta(\mathbf{F})>0$.
\end{lemma}

\begin{proof}
By Remark~\ref{remark:Kento-divisorial} and Lemma~\ref{lemma:2-22-Eff}, it is enough to show that
$\beta(\widetilde{Q})$, $\beta(\widetilde{H}_C)$, $\beta(\widetilde{E})$ are positive.
We will do this using the notations introduced in Section~\ref{section:Abban-Zhuang}.

We start with $\widetilde{Q}$. Let $Y=\widetilde{X}$ and $\mathscr{S}=\widetilde{Q}$. Then
$-K_{\widetilde{X}}-u\mathscr{S}\sim_{\mathbb{R}}(2-u)\mathscr{S}+\widetilde{E}+\widetilde{F}_1+\widetilde{F}_2$.
This shows that $\tau=2$. Moreover, we have
$$
P(u)\sim_{\mathbb{R}}\begin{cases}
(2-u)\mathscr{S}+\widetilde{E}+\widetilde{F}_1+\widetilde{F}_2& \quad \text{if } 0\leqslant u\leqslant 1,\\
(2-u)\big(\mathscr{S}+\widetilde{E}+\widetilde{F}_1+\widetilde{F}_2\big)& \quad \text{if } 1\leqslant u\leqslant 2.
\end{cases}
$$
If $u\in[0,1]$, then $N(u)=0$. If $u\in[1,2]$, then $N(u)=(u-1)(\widetilde{E}+\widetilde{F}_1+\widetilde{F}_2)$.
Then
$$
\big(P(u)\big)^3=\begin{cases}
2u^3-6u^2-18u+30& \quad \text{if } 0\leqslant u\leqslant 1,\\
8(2-u)^3& \quad \text{if } 1\leqslant u\leqslant 2.
\end{cases}
$$
Now, integrating $(P(u))^3$, we get $S_Y(\mathscr{S})=\frac{43}{60}$, so that $\beta(\widetilde{Q})=\frac{17}{60}>0$.

Now we deal with $\widetilde{H}_C$. Set $Y=\widetilde{X}$ and $\mathscr{S}=\widetilde{H}_C$.
Then $-K_{\widetilde{X}}-u\mathscr{S}\sim_{\mathbb{R}}(2-u)\mathscr{S}+2\widetilde{E}+\widetilde{Q}$.
This gives $\tau=2$, because $2\widetilde{E}+\widetilde{Q}$ is not big. Moreover, we have
$$
P(u)\sim_{\mathbb{R}}
\begin{cases}
(2-u)\mathscr{S}+2\widetilde{E}+\widetilde{Q}& \quad \text{if } 0\leqslant u\leqslant 1,\\
(2-u)\mathscr{S}+(3-u)\widetilde{E}+\widetilde{Q}& \quad \text{if } 1\leqslant u\leqslant 2.
\end{cases}
$$
If $u\in[0,1]$, then $N(u)=0$. If $u\in[1,2]$, then $N(u)=(u-1)\widetilde{E}$.
We compute
$$
\big(P(u)\big)^3=\begin{cases}
u^3-6u^2-12u+30& \quad \text{if } 0\leqslant u\leqslant 1,\\
(2-u)(u^2-10u+22)& \quad \text{if } 1\leqslant u\leqslant 2,
\end{cases}
$$
which gives $S_Y(\mathscr{S})=\frac{11}{12}$, so that $\beta(\widetilde{Q})=\frac{1}{12}>0$.

Finally, we set $Y=\overline{X}$ and $\mathscr{S}=\overline{E}$. Then
$-K_{\overline{X}}-u\mathscr{S}\sim_{\mathbb{R}} (2-u)\mathscr{S}+2\overline{H}_C+\overline{Q}$.
This shows that $\tau=2$, because $2\overline{H}_C+\overline{Q}$ is not big.
Moreover, we have
$$
P(u)\sim_{\mathbb{R}}\begin{cases}
(2-u)\mathscr{S}+2\overline{H}_C+\overline{Q} & \quad \text{ if } 0\leqslant u \leqslant 1,\\
(2-u)\big(\mathscr{S}+\overline{Q}+2\overline{H}_C\big) & \quad \text{ if } 1\leqslant u \leqslant 2.
\end{cases}
$$
If $u\in[0,1]$, then $N(u)=0$. If $u\in[1,2]$, then $N(u)=(u-1)\overline{Q}+2(u-1)\overline{H}_C$. Then
$$
\big(P(u)\big)^3=\begin{cases}
6u^{3}-6u^{2}-24u+30& \quad \text{if } 0\leqslant u\leqslant 1,\\
6(2-u)^3& \quad \text{if } 1\leqslant u\leqslant 2,
\end{cases}
$$
which gives $S_Y(\mathscr{S})=\frac{19}{30}$ and $\beta(\widetilde{E})=\beta(\overline{E})=\frac{11}{30}>0$.
\end{proof}

We now show that  $\beta(\mathbf{F})>0$ for $\mathbf{F}$ a $G$-invariant prime divisor with small center on $X$.

\begin{lemma}
\label{lemma:2-22-curve-in-HC}
Suppose that $\widetilde{Z}$ is a $G$-invariant irreducible curve in $\widetilde{H}_C$. Then $\beta(\mathbf{F})>0$.
\end{lemma}

\begin{proof}
The morphism $\gamma\circ\phi$ induces a birational morphism $\widetilde{H}_C\to H_C$,
which is a blowup of the points $H_C\cap L_1$ and $H_C\cap L_2$.
Set $\widetilde{f}_1=\widetilde{F}_1\vert_{\widetilde{H}_C}$ and $\widetilde{f}_2=\widetilde{F}_2\vert_{\widetilde{H}_C}$.
Then~$\widetilde{f}_1$ and $\widetilde{f}_2$ are exceptional curves of the morphism  $\widetilde{H}_C\to H_C$.
Let $\widetilde{l}$ be the third $(-1)$-curve in $\widetilde{H}_C$, set $h=\widetilde{H}\vert_{\widetilde{H}_C}$, and set $\widetilde{C}=\widetilde{E}\vert_{\widetilde{H}_C}$.
Then $\gamma\circ\phi(\widetilde{l})=l$, so that  $\widetilde{l}\sim h-\widetilde{f}_1-\widetilde{f}_2$.
By Lemmas~\ref{lemma:2-22-curves} and \ref{lemma:2-22-curves-E}, we have the following possible cases:
\begin{itemize}
\item $\pi(Z)=l$ and $\widetilde{Z}=\widetilde{l}$,

\item $\pi(Z)=C_1=C$ and $\widetilde{Z}=\widetilde{C}\sim 2h-\widetilde{f}_1-\widetilde{f}_2$,

\item $\pi(Z)=C_r$ with $r\ne 1$, $\widetilde{Z}\not\subset \widetilde{E}$ and $\widetilde{Z}\sim 2h-\widetilde{f}_1-\widetilde{f}_2$.
\end{itemize}
Set $Y=\widetilde{X}$, $\mathscr{S}=\widetilde{H}_C$, $\mathscr{C}=\widetilde{Z}$.
Then it follows from the proof of Lemma~\ref{lemma:2-22-divisorial} that
$$
P(u)\big\vert_{\mathscr{S}}\sim_{\mathbb{R}}\begin{cases}
(2+u)h-u(\widetilde{f}_1+\widetilde{f}_2) & \quad \text{if } 0\leqslant u\leqslant 1,\\
 (4-u)h-\widetilde{f}_1-\widetilde{f}_2 & \quad \text{if } 1\leqslant u\leqslant 2,
\end{cases}
$$
and
$$
N(u)\big\vert_{\mathscr{S}}=\begin{cases}
	0& \quad \text{if } 0\leqslant u\leqslant 1,\\
	(u-1)\widetilde{C}& \quad \text{if } 1\leqslant u\leqslant 2.
\end{cases}
$$
We know from the proof of Lemma~\ref{lemma:2-22-divisorial} that $S_Y(\mathscr{S})=\frac{11}{12}<1$.
Let us compute $S(W^\mathscr{S}_{\bullet,\bullet};\mathscr{C})$.

Suppose that $\widetilde{Z}=\widetilde{l}$.
If $0\leqslant u \leqslant 1$, then $t(u)=2+u$.
If $1\leqslant u \leqslant 2$, thent $t(u)=4-u$.
Moreover, if $0\leqslant u \leqslant 1$, then
$$
P(u,v)\sim_{\mathbb{R}}\begin{cases}
(2+u-v)h-(u-v)(\widetilde{f}_1+\widetilde{f}_2)& \quad \text{if } 0\leqslant v \leqslant u,\\
(2+u-v) h & \quad \text{if } u\leqslant v \leqslant 2+u,
\end{cases}
$$
and
$$
N(u,v)=\begin{cases}
0& \quad \text{if } 0\leqslant v \leqslant u,\\
(v-u)(\widetilde{f}_1+\widetilde{f}_2)& \quad \text{if } u\leqslant v \leqslant 2+u.
\end{cases}
$$
Similarly, if $1\leqslant u \leqslant 2$, then
$$
P(u,v)\sim_{\mathbb{R}}\begin{cases}
(4-u-v)h-(1-v)(\widetilde{f}_1+\widetilde{f}_2)& \quad \text{if } 0\leqslant v \leqslant 1,\\
(4-u-v) h & \quad \text{if } 1\leqslant v \leqslant 4-u,
\end{cases}
$$
and
$$
N(u,v) = \begin{cases}
0& \quad \text{if } 0\leqslant v \leqslant 1,\\
(v-1)(\widetilde{f}_1+\widetilde{f}_2)& \quad \text{if } 1\leqslant v \leqslant 4-u.
\end{cases}
$$
This gives
\begin{multline*}
S(W^\mathscr{S}_{\bullet,\bullet};\mathscr{C})=\frac{1}{10}\int\limits_{0}^1\int\limits_{0}^{u}4-u^{2}+2uv-v^{2}+4u-4vdvdu+\frac{1}{10}\int\limits_{0}^1\int\limits_{u}^{2+u}(2+u-v)^2dvdu+\\
+\frac{1}{10}\int\limits_{1}^2\int\limits_{0}^{1}u^{2}+2uv-v^{2}-8u-4v+14dvdu+\frac{1}{10}\int\limits_{1}^2\int\limits_{1}^{4-u}(u +v-4)^{2}dvdu=1.
\end{multline*}
Hence, it follows from Remark~\ref{remark:Kento-curve-strict} that $\beta(\mathbf{F})>0$ in the case when $\widetilde{Z}=\widetilde{l}$.

We may assume that $\pi(Z)=C_r$. Then $\widetilde{Z}\sim 2h-\widetilde{f}_1-\widetilde{f}_2$.
If $0\leqslant u \leqslant 1$, then $t(u)=\frac{2+u}{2}$.
Similarly, if $1\leqslant u \leqslant 2$, then $t(u)=\frac{4-u}{2}$.
Moreover, if $0\leqslant u \leqslant 1$, then
$$
P(u,v)\sim_{\mathbb{R}}\begin{cases}
(2+u-2v)h-(u-v)(\widetilde{f}_1+\widetilde{f}_2) & \quad \text{if } 0\leqslant v \leqslant u,\\
(2+u -2v) h& \quad \text{if } u\leqslant v \leqslant \frac{2+u}{2},
\end{cases}
$$
and
$$
N(u,v)=\begin{cases}
0& \quad \text{if } 0\leqslant v \leqslant u,\\
(v-u)(\widetilde{f}_1+\widetilde{f}_2)& \quad \text{if } u\leqslant v \leqslant \frac{2+u}{2}.
\end{cases}
$$
Similarly, if $1\leqslant u \leqslant 2$, then
$$
P(u,v)\sim_{\mathbb{R}}\begin{cases}
(4-u-2v)h-(1-v)(\widetilde{f}_1+\widetilde{f}_2)& \quad \text{if } 0\leqslant v \leqslant 1,\\
(4-u-2v)h& \quad \text{if } 1\leqslant v \leqslant \frac{4+u}{2},
\end{cases}
$$
and
$$
N(u,v)=\begin{cases}
0& \quad \text{if } 0\leqslant v \leqslant 1,\\
(v-1)(\widetilde{f}_1+\widetilde{f}_2)& \quad \text{if } 1\leqslant v \leqslant \frac{2+u}{2}.
\end{cases}
$$
Therefore, if $\widetilde{Z}=\widetilde{C}$, then $S(W^\mathscr{S}_{\bullet,\bullet};\mathscr{C})$ can be computed as follows:
\begin{multline*}\hspace*{-1cm}
\frac{1}{10}\int\limits_{1}^{2}(u-1)(u^{2}-8u+14)du+\frac{1}{10}\int\limits_{0}^1\int\limits_{0}^{u}4-u^{2}+2v^{2}+4u-8vdvdu+\frac{1}{10}\int\limits_{0}^1\int\limits_{u}^{\frac{2+u}{2}}(2+u-2v)^2dvdu+\\
+\frac{1}{10}\int\limits_{1}^2\int\limits_{0}^{1}u^{2}+4uv+2v^{2}-8u-12v+14dvdu+\frac{1}{10}\int\limits_{1}^2\int\limits_{1}^{\frac{4-u}{2}}(u+2v-4)^2dvdu=\frac{53}{80}<1.
\end{multline*}
Similarly, if $\widetilde{Z}\ne\widetilde{C}$, then $S(W^\mathscr{S}_{\bullet,\bullet};\mathscr{C})=\frac{39}{80}<1$.
Then $\beta(\mathbf{F})>0$ by \eqref{equation:Kento-curve}.
\end{proof}

Using computations made in the proof of Lemma~\ref{lemma:2-22-curve-in-HC}, we obtain the following result:

\begin{lemma}
\label{lemma:2-22-point-HC}
Suppose that $\pi(Z)$ contains $O$. Then $\beta(\mathbf{F})>0$.
\end{lemma}

\begin{proof}
Let us use notations introduced in the proof of Lemma~\ref{lemma:2-22-curve-in-HC}.
First, we let $Y=\widetilde{X}$.
Let $P$ be the preimage on $Y$ of the point $O$.
Then $P$ is the unique $G$-fixed point in $\widetilde{H}_C$.

As in the proof of Lemma~\ref{lemma:2-22-curve-in-HC}, we choose $\mathscr{S}=\widetilde{H}_C$,
and we choose $\mathscr{C}$ to be the curve in the pencil $|h-\widetilde{f}_1|$ that contains $P$.
Since $S_Y(\mathscr{S})=\frac{11}{12}$ (see the proof of Lemma~\ref{lemma:2-22-divisorial}),
it follows from \eqref{equation:Kento-point} that $\beta(\mathbf{F})>0$
provided that $S(W^\mathscr{S}_{\bullet,\bullet};\mathscr{C})<1$ and $S(W_{\bullet, \bullet,\bullet}^{\mathscr{S},\mathscr{C}};P)<1$.

Let us compute $S(W^\mathscr{S}_{\bullet,\bullet};\mathscr{C})$ and $S(W_{\bullet, \bullet,\bullet}^{\mathscr{S},\mathscr{C}};P)$.
If $0\leqslant u \leqslant 1$, then $t(u)=2$. If $1\leqslant u \leqslant 2$, then $t(u)=3-u$.
Moreover, if $0\leqslant u\leqslant 1$, then
$$
P(u,v)\sim_{\mathbb{R}}\begin{cases}
(2+u-v)h-(u-v)\widetilde{f}_1-u\widetilde{f}_2&\quad \text{ if } 0\leqslant v \leqslant u,\\
(2+u-v)h-u\widetilde{f}_2&\quad \text{ if } u\leqslant v \leqslant 2,
\end{cases}
$$
and
$$
N(u,v)=\begin{cases}
0& \quad \text{if } 0\leqslant v \leqslant u,\\
(v-u)\widetilde{f}_1& \quad \text{if } u\leqslant v \leqslant 2,
\end{cases}
$$
which gives
$$
\big(P(u,v)\big)^2=\begin{cases}
4-u^2+4u-4v&\quad \text{ if } 0\leqslant v \leqslant u,\\
(2-v)(2+2u-v)&\quad \text{ if } u\leqslant v \leqslant 2,
\end{cases}
$$
and
$$
P(u,v)\cdot\mathscr{C}=\begin{cases}
2&\quad \text{ if } 0\leqslant v \leqslant u,\\
2+u-v&\quad \text{ if } u\leqslant v \leqslant 2.
\end{cases}
$$
Similarly, if $1\leqslant u \leqslant 2$, then
$$
P(u,v)\sim_{\mathbb{R}}\begin{cases}
(4-u-v)h-(1-v)\widetilde{f}_1-\widetilde{f}_2&\quad \text{if } 0\leqslant v \leqslant 1,\\
(4-u-v)h-\widetilde{f}_2   &\quad \text{if } 1\leqslant v \leqslant 3-u,
\end{cases}
$$
and
$$
N(u,v)=\begin{cases}
0& \quad \text{if } 0\leqslant v \leqslant 1,\\
(v-1)\widetilde{f}_1& \quad \text{if } 1\leqslant 3-u \leqslant 3-u,
\end{cases}
$$
which implies that
$$
\big(P(u,v)\big)^2=\begin{cases}
u^2+2uv-8u-6v+14&\quad \text{if } 0\leqslant v \leqslant 1,\\
(3-u-v)(5-u-v)   &\quad \text{if } 1\leqslant v \leqslant 3-u,
\end{cases}
$$
and
$$
P(u,v)\cdot\mathscr{C}=\begin{cases}
3-u&\quad \text{if } 0\leqslant v \leqslant 1,\\
4-u-v&\quad \text{if } 1\leqslant v \leqslant 3-u.
\end{cases}
$$

Observe that $\mathscr{C}$ is not contained in the support of the divisor $N(u)$ for every $u\in[0,2]$,
and $P$ is not contained in the support of the divisor $N(u,v)$ for $u\in[0,2]$ and $v\in[0,t(u)]$.
Now, by integrating we get $S(W^\mathscr{S}_{\bullet,\bullet};\mathscr{C})=S(W_{\bullet, \bullet,\bullet}^{\mathscr{S},\mathscr{C}};P)=\frac{47}{60}<1$,
so $\beta(\mathbf{F})>0$ by \eqref{equation:Kento-point}.
\end{proof}

Recall that $H_{C^\prime}$ is the plane in $\mathbb{P}^3$ that contains $O^\prime$ and $C^\prime_r$ for every $r\in\mathbb{C}^\ast$.

\begin{lemma}
\label{lemma:2-22-curve-HC-prime}
Suppose that $\pi(Z)=C^\prime_r$ for some $r\in\mathbb{C}^\ast$. Then $\beta(\mathbf{F})>0$.
\end{lemma}

\begin{proof}
As above, we use notations introduced in Section~\ref{section:Abban-Zhuang}.
Let $Y=\widetilde{X}$, and let $\mathscr{S}$ be the proper transform on $Y$ of the plane $H_{C^\prime}$.
Then $-K_{\widetilde{X}}-u\mathscr{S}\sim_{\mathbb{R}}\widetilde{Q}+(2-u)\mathscr{S}$.
This gives $\tau=2$.
If $u\in[0,1]$, then $N(u)=0$. If $u\in[1,2]$, then $N(u)=(u-1)\widetilde{Q}$.
Thus, we have
$$
P(u)\sim_{\mathbb{R}}\begin{cases}
\widetilde{Q} +(2-u)\mathscr{S}&\quad\text{ if }u\in[0,1],\\
(2-u)(\widetilde{Q}+\mathscr{S}&\quad\text{ if }u\in[1,2].
\end{cases}
$$
Integrating, we get $S_{Y}(\mathscr{S})=\frac{17}{30}$.

Set $\widetilde{C}^\prime=\widetilde{Q}\vert_{\mathscr{S}}$.
Then $\widetilde{C}^\prime$ is a smooth irreducible $G$-invariant curve,
and
$$
N(u)=
\begin{cases}
0&\quad\text{ if } u\in[0,1],\\
(u-1)\widetilde{C}^\prime&\quad\text{ if } u\in[1,2].
\end{cases}
$$
To describe $P(u)\vert_{\mathscr{S}}$ explicitly, we have to say few words about the surface $\mathscr{S}$.

Set $P_1=H_{C^\prime}\cap L_1$ and $P_2=H_{C^\prime}\cap L_2$.
Recall that $l$ is the line containing  $P_1$ and $P_2$.
Let $\mathcal{P}$ be the pencil on $H_{C^\prime}$  generated by $2l$ and $C^\prime$,
let $l_1$ and $l_2$ be the lines in  $H_{C^\prime}$ that are tangent to $C^\prime$ at the points $P_1$ and $P_2$, respectively.
Then
\begin{itemize}
\item the base locus of the pencil $\mathcal{P}$ consists of the points $P_1$ and $P_2$,
\item the pencil $\mathcal{P}$ contains $l_1+l_2$ and the conic $C_r^\prime$ for every $r\in\mathbb{C}^\ast$,
\item the conics $2l$ and $l_1+l_2$ are the only singular curves in $\mathcal{P}$.
\end{itemize}
The morphism $\gamma\circ\phi$ induces a birational morphism $\xi\colon \mathscr{S}\to H_{C^\prime}$
that resolves the base locus of the pencil $\mathcal{P}$.
The morphism $\xi$ is a composition of $4$ blowups such that we have the following $G$-equivariant commutative diagram:
$$
\xymatrix@R=6mm{
&\mathscr{S}\ar[dl]_{\xi}\ar[dr]&\\
 H_{C^\prime}\ar@{-->}[rr]&&\mathbb{P}^1}
$$
where $H_{C^\prime}\dasharrow\mathbb{P}^1$ is the rational map given by $\mathcal{P}$,
and $\mathscr{S}\to\mathbb{P}^1$ is a surjective morphism.
The birational morphism $\xi$ is a composition of the blowup of the points $P_1$ and $P_2$
with the subsequent blowup of two points in the exceptional curves contained in the proper transforms of    $l_1$ and $l_2$.
Note that $\mathscr{S}$ is a weak del Pezzo surface of degree five.

We have $\widetilde{E}\vert_{\mathscr{S}}=\widetilde{g}_1+\widetilde{g}_2$,
where $\widetilde{g}_1$ and $\widetilde{g}_2$ are two irreducible $\xi$-exceptional $(-1)$-curves such that $\xi(\widetilde{g}_1)=P_1$  and $\xi(\widetilde{g}_2)=P_2$.
Let $\widetilde{f}_1$ and $\widetilde{f}_2$ be the remaining $\xi$-exceptional curves
that are mapped to the points $P_1$ and $P_2$, respectively,
and let $\widetilde{l}$, $\widetilde{l}_1$, $\widetilde{l}_2$, $\widetilde{C}^\prime_r$
be the proper transforms on $\mathscr{S}$ of the curves $l$, $l_1$, $l_2$, $C_r^\prime$, respectively.
Then $\widetilde{C}^\prime=\widetilde{C}^\prime_1$ and
$2\widetilde{l}+\widetilde{f}_1+\widetilde{f}_2\sim\widetilde{l}_1+\widetilde{l}_2\sim  \widetilde{C}^\prime\sim \widetilde{C}^\prime_r$
for every $r\in\mathbb{C}^\ast$,
ant the curves
$\widetilde{f}_1$, $\widetilde{f}_2$, $\widetilde{g}_1$, $\widetilde{g}_2$, $\widetilde{l}$, $\widetilde{l}_1$, $\widetilde{l}_2$
generate the Mori cone $\overline{\mathrm{NE}(\mathscr{S})}$ by \cite[Proposition~8.5]{Coray1988}.

Note that $\widetilde{F}_1\vert_{\mathscr{S}}=\widetilde{f}_1+\widetilde{g}_1$ and $\widetilde{F}_2\vert_{\mathscr{S}}=\widetilde{f}_2+\widetilde{g}_2$.
Using this, we get
\begin{equation}
\label{equation:2-12-H-C-prime}
P(u)\big\vert_{\mathscr{S}}\sim_{\mathbb{R}}
\begin{cases}
\frac{4-u}{2}\widetilde{C}^\prime+\frac{2-u}{2}\big(\widetilde{f}_1+\widetilde{f}_2\big)+(2-u)\big(\widetilde{g}_1+\widetilde{g}_2\big)  & \quad \text{if } 0\leqslant u\leqslant 1,\\
\frac{6-3u}{2}\widetilde{C}^\prime+\frac{2-u}{2}\big(\widetilde{f}_1+\widetilde{f}_2\big)+(2-u)\big(\widetilde{g}_1+\widetilde{g}_2\big)  & \quad \text{if } 1\leqslant u\leqslant 2.
\end{cases}
\end{equation}

Set $\mathscr{C}=\widetilde{Z}$.
Let us compute $S(W^\mathscr{S}_{\bullet,\bullet};\mathscr{C})$.
Recall that $\mathscr{C}\sim \widetilde{C}^\prime$.
Then \eqref{equation:2-12-H-C-prime} gives
$$
t(u)=\begin{cases}
\frac{4-u}{2}& \quad \text{if } 0\leqslant u \leqslant 1,\\
\frac{6-3u}{2} & \quad \text{if } 1\leqslant u \leqslant 2.
\end{cases}
$$
Moreover, if $0\leqslant u\leqslant 1$, then \eqref{equation:2-12-H-C-prime} gives
$$
P(u,v)\sim_{\mathbb{R}}
\begin{cases}
\frac{4-u-2v}{2}\mathscr{C}+(2-u)\big(\widetilde{f}_1+\widetilde{f}_2\big)+\frac{2-u}{2}\big(\widetilde{g}_1+\widetilde{g}_2\big) &\quad\text{ if } 0\leqslant v \leqslant 1,\\
\frac{4-u-2v}{2}\big(\mathscr{C}+\widetilde{f}_1+\widetilde{f}_2+\widetilde{g}_1+\widetilde{g}_2\big) &\quad\text{ if } 1\leqslant v\leqslant \frac{4-u}{2},
\end{cases}
$$
and
$$
N(u,v)=\begin{cases}
0&\quad\text{ if } 0\leqslant v\leqslant 1,\\
(v-1)(\widetilde{f}_1+\widetilde{f}_2+2\widetilde{g}_1+2\widetilde{g}_2)&\quad\text{ if } 1\leqslant v\leqslant\frac{4-u}{2}.
\end{cases}
$$
Similarly, if $1\leqslant u\leqslant 2$, then \eqref{equation:2-12-H-C-prime}  gives
$$
P(u,v)\sim_{\mathbb{R}}\begin{cases}
\frac{6-3u-2v}{2}\mathscr{C}+\frac{2-u}{2}\big(\widetilde{f}_1+\widetilde{f}_2\big)+(2-u)\big(\widetilde{g}_1+\widetilde{g}_2\big)&\quad\text{ if } 0\leqslant v\leqslant 2-u,\\
\frac{6-3u-2v}{2}\big(\mathscr{C}+\widetilde{f}_1+\widetilde{f}_2+2\widetilde{g}_1+2\widetilde{g}_2\big)&\quad\text{ if } 2-u\leqslant v\leqslant\frac{6-3u}{2},
\end{cases}
$$
and
$$
N(u,v)=\begin{cases}
0&\quad\text{ if } 0\leqslant v\leqslant 2-u,\\
(v+u-2)\big(\widetilde{f}_1+\widetilde{f}_2+2\widetilde{g}_1+2\widetilde{g}_2\big)&\quad\text{ if } 2-u\leqslant v\leqslant\frac{6-3u}{2}.
\end{cases}
$$
Therefore, if $\mathscr{C}=\widetilde{C}^\prime$, then we compute $S(W^\mathscr{S}_{\bullet,\bullet};\mathscr{C})$ as follows:
\begin{multline*}\hspace*{-0.7cm}
\frac{1}{10}\int\limits_{1}^{2}(5u^2-20u+20)(u-1)du+\frac{1}{10}\int\limits_{0}^{1}\int\limits_{0}^{1}(2-u)(6-u-4v)dvdu+
\frac{1}{10}\int\limits_{0}^{1}\int\limits_{1}^{\frac{4-u}{2}}(4-u-2v)^2dvdu+\\
+\frac{1}{10}\int\limits_{1}^{2}\int\limits_{0}^{2-u}(2-u)(10-5u-4v)dvdu+\frac{1}{10}\int\limits_{1}^{2}\int\limits_{2-u}^{\frac{6-3u}{2}}(6-3u-2v)^2dvdu=\frac{43}{60},
\end{multline*}
If $\widetilde{Z}\ne\widetilde{C}^\prime$, similar computations give $S(W^\mathscr{S}_{\bullet,\bullet};\mathscr{C})=\frac{27}{40}$,
since $\widetilde{Z}\not\subset\mathrm{Supp}(N(u))$ for $u\in[0,2]$.
Hence, it follows from \eqref{equation:Kento-curve} that $\beta(\mathbf{F})>0$, because $S_{Y}(\mathscr{S})=\frac{17}{30}<1$.
\end{proof}

The proof of Lemma~\ref{lemma:2-22-curve-HC-prime} implies the following result:

\begin{lemma}
\label{lemma:2-22-O-prime}
Suppose that $\pi(Z)=O^\prime$. Then $\beta(\mathbf{F})>0$.
\end{lemma}

\begin{proof}
Let us use all assumptions and notations introduced in the proof of Lemma~\ref{lemma:2-22-curve-HC-prime} with one exception:
now we let $\mathscr{C}=\widetilde{l}_1$.
Set $P=\widetilde{l}_1\cap \widetilde{l}_2$. Then $P=\widetilde{Z}$ and $\gamma\circ\phi(P)=O^\prime$.

Since $\widetilde{C}^\prime\sim\widetilde{l}_1+\widetilde{l}_2$
and $\widetilde{l}_1+\widetilde{f}_1+2\widetilde{g}_1\sim\widetilde{l}_2+\widetilde{f}_2+2\widetilde{g}_2$,
it follows from \eqref{equation:2-12-H-C-prime} that
$$
P(u)\big\vert_{\mathscr{S}}-v\mathscr{C}\sim_{\mathbb{R}}
\begin{cases}
(3-u-v)\mathscr{C}+\widetilde{l}_2+(2-u)\widetilde{f}_1+(4-2u)\widetilde{g}_1 &\quad  \text{if } 0\leqslant u\leqslant 1,\\
(4-2u-v)\mathscr{C}+(2-u)\widetilde{l}_2+(2-u)\widetilde{f}_1+(4-2u)\widetilde{g}_1  &\quad  \text{if } 1\leqslant u\leqslant 2.
\end{cases}
$$
If $u\in[0,1]$, then $t(u)=3-u$.
If $u\in[1,2]$, then $t(u)=4-2u$.
If $u\in[0,1]$, then
$$
P(u,v)\sim_{\mathbb{R}}\begin{cases}
(3-u-v)\mathscr{C}+\widetilde{l}_2+(2-u)\widetilde{f}_1+(4-2u)\widetilde{g}_1&\quad \text{ if } 0\leqslant v \leqslant 1,\\
(3-u-v)\big(\mathscr{C}+\widetilde{f}_1+2\widetilde{g}_1\big)+\widetilde{l}_2&\quad \text{ if } 1\leqslant v \leqslant 2-u,\\
(3-u-v)\big(\mathscr{C}+\widetilde{l}_2+\widetilde{f}_1+2\widetilde{g}_1\big)&\quad \text{ if } 2-u\leqslant v \leqslant 3-u,
\end{cases}
$$
and
$$
N(u,v)=\begin{cases}
0&\quad \text{ if } 0\leqslant v \leqslant 1,\\
(v-1)\big(\widetilde{f}_1+2\widetilde{g}_1\big) &\quad \text{ if } 1\leqslant v \leqslant 2-u,\\
(v-1)\big(\widetilde{f}_1+2\widetilde{g}_1\big)+(v+u-2)\widetilde{l}_2 &\quad \text{ if } 2-u\leqslant v \leqslant 3-u,
\end{cases}
$$
which gives
$$
\big(P(u,v)\big)^2=\begin{cases}
u^2+2uv-v^2-8u-4v+12&\quad \text{ if } 0\leqslant v \leqslant 1,\\
u^2+2uv+v^2-8u-8v+14&\quad \text{ if } 1\leqslant v \leqslant 2-u,\\
2(3-u-v)^2 &\quad \text{ if } 2-u\leqslant v \leqslant 3-u,
\end{cases}
$$
and
$$
P(u,v)\cdot\mathscr{C}=\begin{cases}
2-u+v&\quad \text{ if } 0\leqslant v \leqslant 1,\\
4-u-v&\quad \text{ if } 1\leqslant v \leqslant 2-u,\\
6-2u-2v &\quad \text{ if } 2-u\leqslant v \leqslant 3-u.
\end{cases}
$$
Similarly, if $u\in[1,2]$, then
$$
P(u,v)\sim_{\mathbb{R}}\begin{cases}
(4-2u-v)\mathscr{C}+(2-u)\widetilde{l}_2+(2-u)\widetilde{f}_1+(4-2u)\widetilde{g}_1&\quad \text{if } 0\leqslant v \leqslant 2-u,\\
(4-2u-v)\big(\mathscr{C}+\widetilde{l}_2+\widetilde{f}_1+2\widetilde{g}_1\big)&\quad \text{if } 2-u\leqslant v \leqslant 4-2u,
\end{cases}
$$
and
$$
N(u,v)=\begin{cases}
0&\quad \text{if } 0\leqslant v \leqslant 2-u,\\
(v+u-2)\big(\widetilde{l}_2+\widetilde{f}_1+2\widetilde{g}_1\big)&\quad \text{if } 2-u\leqslant v \leqslant 4-2u,
\end{cases}
$$
which implies that
$$
\big(P(u,v)\big)^2=\begin{cases}
5u^2+2uv-v^2-20u-4v+20&\quad \text{if } 0\leqslant v \leqslant 2-u,\\
2(4-2u-v)^2&\quad \text{if } 2-u\leqslant v \leqslant 4-2u,
\end{cases}
$$
and
$$
P(u,v)\cdot\mathscr{C}=\begin{cases}
2-u+v&\quad \text{if } 0\leqslant v \leqslant 2-u,\\
8-4u-2v&\quad \text{if } 2-u\leqslant v \leqslant 4-2u.
\end{cases}
$$
Observe that $d(u)=0$ for $u\in[0,2]$, since $\mathscr{C}\not\subset\mathrm{Supp}(N(u))$.
Therefore, integrating, we get $S(W^\mathscr{S}_{\bullet,\bullet};\mathscr{C})=1$.
Similarly, we compute
$$
F_P\big(W_{\bullet,\bullet,\bullet}^{\mathscr{S},\mathscr{C}}\big)=\frac{1}{5}\int\limits_0^1\int\limits_{2-u}^{3-u}(v+u-2)\big(P(u,v)\cdot \mathscr{C}\big)dvdu+\frac{1}{5}\int\limits_1^2\int\limits_{2-u}^{4-2u}(v+u-2)\big(P(u,v)\cdot \mathscr{C}\big)dvdu=\frac{1}{12}
$$
and $S(W_{\bullet, \bullet,\bullet}^{\mathscr{S},\mathscr{C}};P)=1$. Thus, it follows from Remark~\ref{remark:Kento-point-strict} that $\beta(\mathbf{F})>0$, since  $S_{Y}(\mathscr{S})<1$.
\end{proof}

Finally, we prove the following result

\begin{lemma}
\label{lemma:2-22-curve-E}
Suppose that $\overline{Z}$ be a $G$-invariant irreducible curve in $\overline{E}$. Then $\beta(\mathbf{F})>0$.
\end{lemma}

\begin{proof}
We have $\delta(\overline{E})\cong\mathbb{F}_2$, see  the proof of Lemma~\ref{lemma:2-22-curves-E}.
Set $\overline{s}_0=\overline{Q}\cap\overline{E}$ and $\overline{s}=\overline{H}_C\cap \overline{E}$.
Then $\delta(\overline{s})$ is the unique $(-2)$-curve in $\delta(\overline{E})$,
and $\delta(\overline{s}_0)$ is a section of the projection $\delta(\overline{E})\to C$ disjoint from $\delta(\overline{s})$.
The morphism $\delta$ induces a birational map $\xi\colon\overline{E}\to \delta(\overline{E})$ that blows up two points in $\delta(\overline{s}_0)$.

Set $\overline{f}_1=\overline{F}_1\cap\overline{E}$ and $\overline{f}_2=\overline{F}_2\cap\overline{E}$.
Observe that $\overline{f}_1$ and $\overline{f}_2$ are the $\xi$-exceptional curves.
Let $\overline{g}_1$ and $\overline{g}_1$ be the proper transforms on $\overline{E}$ of the fibres of the projection
$\delta(\overline{E})\to C$ that pass through $\xi(\overline{f}_1)$ and $\xi(\overline{f}_2)$, respectively.
The curves $\overline{f}_1$, $\overline{f}_2$, $\overline{g}_1$, $\overline{g}_2$ are $(-1)$-curves,
and the curves $\overline{s}$, $\overline{f}_1$, $\overline{f}_2$, $\overline{g}_1$, $\overline{g}_2$
generates the Mori cone $\overline{\mathrm{NE}(\overline{E})}$.
Note that $\overline{s}_0$ is a $(0)$-curve.

The curves $\overline{s}$ and $\overline{s}_0$ are the only $G$-invariant irreducible curves in $\overline{E}$ by Lemma~\ref{lemma:2-22-curves-E}.
Moreover, if $\overline{Z}=\overline{s}$, then $\beta(\mathbf{F})>0$ by Lemma~\ref{lemma:2-22-curve-in-HC}.
Thus, we may assume that $\overline{Z}=\overline{s}_0$.

Let $Y=\overline{X}$, $\mathscr{S}=\overline{E}$ and $\mathscr{C}=\overline{s}_0$.
Then it follows from the proof of Lemma~\ref{lemma:2-22-divisorial} that
$$
P(u)\big\vert_{\mathscr{S}}-v\mathscr{C}\sim_{\mathbb{R}}\begin{cases}
(1+u-v)\mathscr{C}+(2-u)(\overline{f}_1+\overline{f}_2)+(2-2u)(\overline{g}_1+\overline{g}_2) & \quad \text{if } 0\leqslant u\leqslant 1,\\
(4-2u-v)\mathscr{C}+(2-u)(\overline{f}_1+\overline{f}_2) & \quad \text{if } 1\leqslant u\leqslant 2.
\end{cases}
$$
Moreover, if $u\in[0,1]$, then $t(u)=1+u$ and $N(u)\vert_{\mathscr{S}}=0$.
Furthermore, if $u\in[1,2]$, then $t(u)=4-2u$ and $N(u)\vert_{\mathscr{S}}=(u-1)\mathscr{C}+2(u-1)\overline{s}$.
If $u\in[0,1]$, then
$$
P(u,v) = \begin{cases}
(1+u-v)\mathscr{C}+(2-u)(\overline{f}_1+\overline{f}_2)+(2-2u)(\overline{g}_1+\overline{g}_2)& \quad \text{if } 0\leqslant v \leqslant 1,\\
(1+u-v)\mathscr{C}+(3-u-v)(\overline{f}_1+\overline{f}_2)+(2-2u)(\overline{g}_1+\overline{g}_2)	& \quad \text{if } 1\leqslant v \leqslant 1+u,
\end{cases}
$$
and
$$
N(u,v)=\begin{cases}
0& \quad \text{if } 0\leqslant v \leqslant u,\\
(v-1)(\overline{f}_1+\overline{f}_2)& \quad \text{if } u\leqslant v \leqslant 1+u.
\end{cases}
$$
Similarly, if $u\in[1,2]$, then	
$$
P(u,v) = \begin{cases}
(4-2u-v)\mathscr{C}+(2-u)(\overline{f}_1+\overline{f}_2) & \quad \text{if } 0\leqslant v \leqslant 2-u,\\
(4-2u-v)\big(\mathscr{C}+\overline{f}_1+\overline{f}_2\big) & \quad \text{if } 2-u\leqslant v \leqslant 4-2u,
\end{cases}
$$
and
$$
N(u,v)=\begin{cases}
0& \quad \text{if } 0\leqslant v \leqslant 2-u,\\
(v+u-2)(\overline{f}_1+\overline{f}_2)& \quad \text{if } 2-u\leqslant v \leqslant 4-2u.
\end{cases}
$$
Now, we compute $S(W^\mathscr{S}_{\bullet,\bullet};\mathscr{C})$ as follows:
\begin{multline*}\hspace*{-1.2cm}
\frac{1}{10}\int\limits_{1}^{2} (u-1)(6u^{2}-24 u +24)du+\frac{1}{10}\int\limits_{0}^1\int\limits_{0}^{1}8-6u^{2}+4uv+4u-8vdvdu+
\frac{1}{10}\int\limits_{0}^1\int\limits_{1}^{u+1}2(5-3u-v)(u+1-v)dvdu+\\
+\frac{1}{10}\int\limits_{1}^2\int\limits_{0}^{2-u}6u^{2}+4uv-24u-8v+24dvdu+\frac{1}{10}\int\limits_{1}^2\int\limits_{2-u}^{4-2u}2(4-2u-v)^2dvdu=\frac{43}{60}.
\end{multline*}
But $S_Y(\mathscr{S})=\frac{19}{30}$, see the proof of Lemma~\ref{lemma:2-22-divisorial}.
Then \eqref{equation:Kento-curve} gives $\beta(\mathbf{F})>0$.
\end{proof}

By Lemmas~\ref{lemma:2-22-points}, \ref{lemma:2-22-curves}, \ref{lemma:2-22-divisorial}, \ref{lemma:2-22-curve-in-HC},
\ref{lemma:2-22-point-HC}, \ref{lemma:2-22-curve-HC-prime}, \ref{lemma:2-22-O-prime},
\ref{lemma:2-22-curve-E}, we have $\beta(\mathbf{F})>0$ in all possible cases except maybe when  $\pi(Z)=l^\prime$.
But in this case, we have $O\in\pi(Z)$, so $\beta(\mathbf{F})>0$ by Lemma~\ref{lemma:2-22-point-HC}.
Thus, we conclude that $X$ is K-polystable by \cite[Corollary~4.14]{Zhuang}.

\section{K-polystability of the Fano 3-fold $X_{\infty}$ in Family \ref{example:3-12}}
\label{section:3-12}

Consider the following lines in $\mathbb{P}^3$:
$L=\{x_0=0,x_3=0\}$,
$\ell_1=\{x_1=x_0,x_2=0\}$,
$\ell_2=\{x_1=0,x_2=0\}$,
$\ell_3=\{x_2=x_3,x_1=0\}$,
where $x_0$, $x_1$, $x_2$, $x_3$ are coordinates on~$\mathbb{P}^3$.
Then $L$ is disjoint from $\ell_1$, $\ell_2$, $\ell_3$,
the lines $\ell_1$ and $\ell_3$ are disjoint, $\ell_2\cap \ell_1=[0:0:0:1]$ and $\ell_2\cap \ell_3=[1:0:0:0]$.
Let $\pi\colon X\to \mathbb{P}^3$ be the blowup of the curve $L+\ell_1+\ell_2+\ell_3$,
and let $\chi\colon\mathbb{P}^3\dasharrow \mathbb{P}^1\times\mathbb{P}^2$ be the dominant rational map given by
$$
[x_0:x_1:x_2:x_3]\mapsto\big([x_0:x_3],[x_1(x_0-x_1):x_1x_2:x_2(x_3-x_2)]\big).
$$
Then $\chi$ is undefined along $L\cup\ell_1\cup\ell_2\cup\ell_3$, $\pi$ resolves the indeterminacy of $\chi$,
and there exists a birational morphism
$\eta\colon X\to \mathbb{P}^1\times\mathbb{P}^2$ that fits in the following commutative diagram:
\begin{equation}
\label{equation:3-12-P1-P2}
\xymatrix@R=6mm{
&X\ar[dl]_{\pi}\ar[dr]^{\eta}&\\
\mathbb{P}^3\ar@{-->}[rr]^{\chi}&&\mathbb{P}^1\times\mathbb{P}^2.}
\end{equation}
To describe $\eta$, set $H_{12}=\{x_2=0\}$ and $H_{23}=\{x_1=0\}$, and denote
$$
Q=\big\{x_0x_2+x_1x_3-x_0x_3=0\big\}.
$$
Then $H_{12}$ is the plane containing the lines $\ell_1$ and $\ell_2$,
$H_{23}$ is the plane containing $\ell_2$ and~$\ell_3$,
and $Q$ is the unique smooth quadric in $\mathbb{P}^3$ that contains $L$, $\ell_1$, and $\ell_3$. Further,
$\chi(H_{12})$ is the curve $\mathbb P^1\times \{[1:0:0])\}$,
$\chi(H_{23}) = \mathbb P^1\times \{[0:0:1])\}$, and
$\chi(Q)$ is the curve parametrised as $([u:v],[u^2:uv:v^2])$,
where $[u:v]\in\mathbb{P}^1$.
Therefore, we see that $\eta$ contracts the proper transforms of the surfaces $H_{12}$, $H_{23}$, $Q$
to the curves $\chi(H_{12})$, $\chi(H_{23})$, $\chi(Q)$, respectively.
Note that these curves are contained
in the preimage of the smooth conic $\{xz=y^2\}$ via the  projection $\mathbb{P}^1\times\mathbb{P}^2\to\mathbb{P}^2$,
where $[x:y:z]$ are coordinates on $\mathbb{P}^2$. Observe  also that
$$
\chi(Q)=\big\{uy=vx,vy=uz\big\}\subset\mathbb{P}^1\times\mathbb{P}^2.
$$
Finally, note that $\chi(H_{12})$ and $\chi(H_{23})$ are the fibres of the projection $\mathbb{P}^1\times\mathbb{P}^2\to\mathbb{P}^2$
over the points $[1:0:0]$ and $[0:0:1]$, respectively. Hence, we have the following conclusion.

\begin{corollary}
\label{corollary:3-12-another-model}
The 3-fold $X$ is isomorphic to the 3-fold $X_\infty$ described in Family \ref{example:3-12}.
\end{corollary}
\paragraph*{\bf Description of the automorphism group}
Let $\tau$ be the involution of $\mathbb{P}^3$ defined by
$[x_0:x_1:x_2:x_3]\mapsto[x_3:x_2:x_1:x_0]$,
and let $\Gamma$ be the subgroup in $\mathrm{Aut}(\mathbb{P}^3)$ consisting of automorphisms
$$
[x_0:x_1:x_2:x_3]\mapsto\left[\lambda x_0:\lambda x_1:x_2:x_3\right],
$$
where $\lambda\in\mathbb{C}^\ast$.
Set $G=\langle\tau,\Gamma\rangle$. Then $\Gamma\cong\mathbb{C}^\ast$ and $G\cong\mathbb{C}^\ast\rtimes\mumu_2$.
Since $L+\ell_1+\ell_2+\ell_3$ is $G$-invariant,
the action of the group $G$ lifts to $X$.
Hence, we can identify $G$ with a subgroup~in~$\mathrm{Aut}(X)$.
One can check that $\mathrm{Aut}(X)=G$. Note that  \eqref{equation:3-12-P1-P2} is $G$-equivariant.

\paragraph*{\bf Description of the $G$-invariant loci}
Consider the smooth quadric surface $R=\{x_0x_2=x_1x_3\}$.

\begin{lemma}
\label{lemma:3-12-points}
There are no $G$-fixed point or $G$-invariant plane in $\mathbb{P}^3$.
If $S$ is a $G$-invariant irreducible quadric surface in $\mathbb{P}^3$,
then $S=R$ or
$$
S=\{ax_0x_3+bx_1x_2+c(x_0x_2+x_1x_3)=0\}
$$
for some $[a:b:c]\in\mathbb{P}^2$ such that $ab\ne c^2$.
\end{lemma}

\begin{proof}
Left to the reader.
\end{proof}

For $a\in\mathbb{C}\cup\{\infty\}$, set $L_{a}=\{x_0=ax_1,x_3=ax_2\}\subset \mathbb P^3$; then $L_a$ is a $G$-invariant line lying on $R$.
Note that $L_{0}=L$ and $L_{\infty}=\ell_2$.

\begin{lemma}
\label{lemma:3-12-curves}
If $C\subset \mathbb P^3$ is a $G$-invariant irreducible curve, then $C=L_a$ for some $a\in\mathbb{C}\cup\{\infty\}$.
\end{lemma}

\begin{proof}
See the proof of Lemma~\ref{lemma:2-22-curves}.
\end{proof}

Denote by $\gamma\colon V\to \mathbb{P}^3$ the blowup of the lines $L$, $\ell_1$, $\ell_3$,
and by $\phi\colon\widetilde{X}\to V$ the blowup of the proper transform of the line $\ell_2$.

Let $\varphi\colon W\to \mathbb P^3$ be the blowup of the lines $L$ and $\ell_2$,
and $\delta\colon \overline{X}\to W$ the blowup of the proper transform of the disjoint lines $\ell_1$ and $\ell_3$.
Then we have a $G$-equivariant commutative diagram:
\begin{equation}
\label{equation:3-12-diagram}
\xymatrix{
\widetilde{X}\ar[rr]\ar[d]_{\phi} && X\ar[d]_{\pi} && \overline{X}\ar[ll]\ar[d]^{\delta}\\
V \ar[rr]_{\gamma} && \mathbb{P}^3 && W\ar[ll]^\varphi}
\end{equation}
where $\widetilde{X}\to X$ and $\overline{X}\to X$ are $G$-equivariant small resolutions of the 3-fold~$X$.

Let $E_L$, $E_1$, $E_2$, $E_3$ be the $\pi$-exceptional divisors that are mapped to $L$, $\ell_1$, $\ell_2$, $\ell_3$, respectively,
let $H_L$ be a general plane in $\mathbb{P}^3$ that contains $L$, let $H_2$ be a general plane in $\mathbb{P}^3$ that contains $\ell_2$,
let $H$ be a general plane in $\mathbb{P}^3$,
let $\overline{E}_L$, $\overline{E}_1$, $\overline{E}_2$, $\overline{E}_3$, $\overline{Q}$, $\overline{R}$,
$\overline{H}_{12}$, $\overline{H}_{23}$, $\overline{H}_L$, $\overline{H}_2$, $\overline{H}$
be the proper transforms on $\overline{X}$ of the surfaces
$E_L$, $E_1$, $E_2$, $E_3$, $Q$, $R$, $H_{12}$, $H_{23}$, $H_L$, $H_2$, $H$, respectively.
Then $\overline{H}$, $\overline{E}_L$, $\overline{E}_1$, $\overline{E}_2$, $\overline{E}_3$
generate  $\mathrm{Pic}(\overline{X})$, and their intersections can be described as follows:
$\overline{H}^3=1$, $\overline{E}_L^3=\overline{E}_2^3=-2$,
$\overline{E}_1^3=\overline{E}_3^3=-1$, $\overline{E}_L^2\cdot \overline{H}=\overline{E}_1^2\cdot \overline{H}=\overline{E}_2^2\cdot \overline{H}=\overline{E}_3^2\cdot \overline{H}=\overline{E}_2\cdot \overline{E}_3^2=\overline{E}_2\cdot \overline{E}_1^2=-1$,
and other triple intersections are zero.
Note that $-K_{\overline{X}}\sim 4\overline{H}-\overline{E}_L-\overline{E}_1-\overline{E}_2-\overline{E}_3$ and
\begin{align*}
\overline{Q}&\sim 2\overline{H}-\overline{E}_L-\overline{E}_1-\overline{E}_3, & \overline{H}_{12}&\sim \overline{H}-\overline{E}_1-\overline{E}_2, & \overline{H}_L&\sim \overline{H}-\overline{E}_L, \\
\overline{R}&\sim 2\overline{H}-\overline{E}_L-\overline{E}_2, & \overline{H}_{23}&\sim \overline{H}-\overline{E}_2-\overline{E}_3, & \overline{H}_2&\sim \overline{H}-\overline{E}_L.
\end{align*}
Note also that $\overline{E}_L$, $\overline{E}_2$, $\overline{E}_1+\overline{E}_3$,
$\overline{Q}$, $\overline{R}$, $\overline{H}_{12}+\overline{H}_{23}$ are $G$-invariant and $G$-irreducible.

\begin{lemma}
\label{lemma:3-12-Eff}
Let $S$ be a $G$-invariant prime divisor in $\overline{X}$.
If $S\ne\overline{E}_L$, then there are non-negative integers $a_1$, $a_2$, $a_3$, $a_4$, $a_5$ such that
$S\sim a_1\overline{E}_2+a_2\overline{Q}+a_3\overline{H}_L+a_4\overline{H}_2+a_5(\overline{H}_{12}+\overline{H}_{23})$.
\end{lemma}

\begin{proof}
We may assume that $S\ne\overline{E}_L$, $S\ne\overline{E}_2$ and $S\ne\overline{Q}$.
Then $\pi(S)$ is a $G$-invariant surface in $\mathbb{P}^3$ of degree $d\geqslant 1$,
so that $S\sim d\overline{H}-m\overline{E}_L-r(\overline{E}_1+\overline{E}_3)-s\overline{E}_2$
for some non-negative integers $m$, $r$, $s$.

Let $\ell$ be a general ruling of the quadric $Q\cong\mathbb{P}^1\times\mathbb{P}^1$
that intersects  $L$, $\ell_1$, $\ell_3$,
let $\overline{\ell}$ be its proper transform on $\overline{X}$.
Then $\overline{\ell}\not\subset S$, which gives $0\leqslant S\cdot\overline{\ell}=d-m-2r$.
Similarly, let $\ell_{12}$ be a general line in the plane $H_{12}$ that passes through the point $H_{12}\cap L$,
and let $\overline{\ell}_{12}$ be its proper transform on $\overline{X}$.
Then $\overline{\ell}_{12}\not\subset S$, which gives $0\leqslant S\cdot\overline{\ell}_{12}=d-m-r-s$.
Thus, if $m\geqslant r$, we can let $a_1=d-m-r-s$, $a_2=r$, $a_3=m-r$, $a_4=d-m-r$, $a_5=0$.
If $m<r$, we can let $a_1=d-2m-s$, $a_2=m$, $a_3=0$, $a_4=d-2r$, $a_5=r-m$.
\end{proof}

\begin{lemma}
\label{lemma:3-12-E2-EL-Q}
Let $S$ be a $G$-invariant prime divisor in $\overline{X}$ such that
$-K_{\overline{X}}\sim_{\mathbb{Q}}\lambda S+\Delta$
for some positive rational number $\lambda>\frac{4}{3}$ and some effective $\mathbb{Q}$-divisor $\Delta$ on the 3-fold $\overline{X}$.
Then $S=\overline{E}_2$, $S=\overline{E}_L$ or $S=\overline{Q}$.
\end{lemma}

\begin{proof}
Suppose that $S\ne\overline{E}_2$ and $S\ne\overline{E}_L$.
Let us show that $S=\overline{Q}$.
Since $S\ne\overline{E}_1+\overline{E}_3$, we see that $\pi(S)$ is a $G$-invariant irreducible surface of degree $d\geqslant 2$,
because $\mathbb{P}^3$ does not contain $G$-invariant planes by Lemma~\ref{lemma:3-12-points}. Then
$S\sim d\overline{H}-m\overline{E}_L-r(\overline{E}_1+\overline{E}_3)-s\overline{E}_2$
for some non-negative integers $m$, $r$, $s$.
Then $4\geqslant \lambda d>\frac{4}{3}d$, so that $d=2$ and
$$
\Delta\sim_{\mathbb{Q}}  (4-2\lambda)\overline{H}+(m\lambda-1)\overline{E}_L+(s\lambda-1)\overline{E}_2+(r\lambda-1)(\overline{E}_1+\overline{E}_3).
$$

Let $\ell$ be a general line in $\mathbb{P}^3$ that intersects the lines $\ell_1$ and $\ell_2$,
and let $\overline{\ell}$ be its proper transform on the 3-fold $\overline{X}$.
Then $\overline{\ell}\not\subset\mathrm{Supp}(\Delta)$, so that
$0\leqslant \Delta\cdot\overline{\ell}=2-2\lambda+r\lambda$,
which implies that $r\ne 0$.
Similarly, intersecting $\Delta$ with the proper transform of a general line in $\mathbb{P}^3$ that intersects $L$ and $\ell_2$,
we see that $(m,s)\ne (0,0)$.

Since $r\ne 0$, the quadric $\pi(S)$ contains $\ell_1$ and $\ell_3$.
Hence, using Lemma~\ref{lemma:3-12-points}, we get
$$
\pi(S)=\big\{ax_0x_3+bx_1x_2-a(x_0x_2+x_1x_3)=0\big\}
$$
for some $[a:b]\in\mathbb{P}^1$ such that $[a:b]\ne [0:1]$ and $[a:b]\ne [1:1]$.
This gives $\ell_2\not\in\pi(S)$, so that $s=0$. Then $m\ne 0$, so that $L\subset \pi(S)$.
Then $[a:b]=[1:0]$ and $\overline{S}=\overline{Q}$.
\end{proof}

We now turn to the proof that $X$ is K-polystable.
Let $\mathbf{F}$ be a $G$-invariant prime divisor over $X$, let $Z$
and $\overline{Z}$ be its centres on $X$ and $Y=\overline{X}$, respectively.
Then it follows from Lemmas~\ref{lemma:3-12-points} and \ref{lemma:3-12-curves} that one of the following four cases holds:
\begin{enumerate}
\item $Z$ is a $G$-invariant irreducible surface,
\item $Z$ is a $G$-invariant irreducible curve in the surface $E_L$,
\item $Z$ is a $G$-invariant irreducible curve in the surface $E_2$,
\item $\pi(Z)=L_a$ for some $a\in\mathbb{C}^\ast$.
\end{enumerate}
By \cite[Corollary~4.14]{Zhuang}, to prove that $X$ is K-polystable, it is enough to show that~\mbox{$\beta(\mathbf{F})>0$}.
We use the assumptions and notations introduced in Section~\ref{section:Abban-Zhuang},
we first consider the case when $Z$ is a surface.

\begin{lemma}
\label{lemma:3-12-divisorial}
Let $S$ be a $G$-invariant prime divisor in $X$.
Then $\beta(S)>0$.
\end{lemma}

\begin{proof}
By Remark~\ref{remark:Kento-divisorial} and Lemma \ref{lemma:3-12-E2-EL-Q},
it is enough to show that $\beta(\overline{E}_2)$, $\beta(\overline{E}_L)$, $\beta(\overline{Q})$ are positive.
Observe that $\beta(\overline{E}_L)>0$ follows from the proof of \cite[Lemma 4.2]{Denisova}.
Nevertheless, let us compute $\beta(\overline{E}_L)$.
We let $\mathscr{S}=\overline{E}_L$.
Then
$$
-K_{Y}-u\mathscr{S}\sim_{\mathbb{R}}
4\overline{H}-(1+u)\mathscr{S}-\overline{E}_1-\overline{E}_2-\overline{E}_3\sim_{\mathbb{R}}
\left(\frac{3}{2}-u\right)\mathscr{S}+
\frac{1}{2}\big(\overline{Q}+\overline{H}_{12}+\overline{H}_{23}\big)+2\overline{H}_L,
$$
Thus, it follows from \eqref{equation:3-12-P1-P2} that $\tau=\frac{3}{2}$. Moreover, we have
$$
P(u)\sim_{\mathbb{R}}\begin{cases}
\left(\frac{3}{2}-u\right)\mathscr{S}+
\frac{1}{2}\big(\overline{Q}+\overline{H}_{12}+\overline{H}_{23}\big)+2\overline{H}_L& \quad \text{if } 0\leqslant u\leqslant 1,\\
\left(\frac{3}{2}-u\right)\big(\mathscr{S}+\overline{Q}+\overline{H}_{12}+\overline{H}_{23}\big)+2\overline{H}_L& \quad \text{if } 1\leqslant u\leqslant \frac{3}{2},
\end{cases}
$$
and
$$
N(u)=\begin{cases}
	0& \quad \text{if } 0\leqslant u\leqslant 1,\\
	(u-1)\big(\overline{Q}+\overline{H}_{12}+\overline{H}_{23}\big)& \quad \text{if } 1\leqslant u\leqslant\frac{3}{2}.
\end{cases}
$$
Now, by integrating $(P(u))^3$ we get $\beta(\overline{E}_L)=1-S_{Y}(\overline{E}_L)=1-\frac{37}{56}=\frac{19}{56}$.

Now, we deal with $\overline{Q}$.
Set $\mathscr{S}=\overline{Q}$.
Since $\overline{Q}\sim 2\overline{H}-\overline{E}_1-\overline{E}_L-\overline{E}_3$, we~have
$$
-K_{Y}-u\mathscr{S}\sim_{\mathbb{R}}
\left(\frac{3}{2}-u\right)\mathscr{S}+\frac{1}{2}\big(\overline{E}_L+\overline{E}_1+\overline{E}_2\big)+\frac{1}{2}\overline{H}_2,
$$
so that $\tau=\frac{3}{2}$. Moreover, if $0\leqslant u\leqslant 1$, then $N(u)=0$.
Similarly, if $1\leqslant u\leqslant\frac{3}{2}$, then $N(u)=(u-1)(\overline{E}_L+\overline{E}_1+\overline{E}_2)$.
Then
$$
P(u)\sim_{\mathbb{R}}\begin{cases}
\left(\frac{3}{2}-u\right)\mathscr{S}+\frac{1}{2}\big(\overline{E}_L+\overline{E}_1+\overline{E}_2\big)+\frac{1}{2}\overline{H}_2& \quad \text{if } 0\leqslant u\leqslant 1,\\
\left(\frac{3}{2}-u\right)\big(\mathscr{S}+\overline{E}_L+\overline{E}_1+\overline{E}_2\big)+\frac{1}{2}\overline{H}_2& \quad \text{if } 1\leqslant u\leqslant \frac{3}{2}.
\end{cases}
$$
Now, by integrating we get $\beta(\overline{Q})=1-S_{Y}(\overline{Q})=1-\frac{129}{224}>0$.

Finally, we proceed to study $\overline{E}_2$.
Let $\mathscr{S}=\overline{E}_2$. Then $\tau=2$, since
$$
-K_{Y}-u\mathscr{S}\sim_{\mathbb{R}}(2-u)\mathscr{S}+\frac{3}{2}\big(\overline{H}_{12}+\overline{H}_{23}\big)+\frac{1}{2}\big(\overline{E}_{1}+\overline{E}_{3}\big).
$$
Moreover, if $u\in[0,1]$, then $N(u)=0$.
If $u\in[1,2]$, then $N(u)=(u-1)(\overline{H}_{12}+\overline{H}_{23})$, so
$$
P(u)\sim_{\mathbb{R}}\begin{cases}
(2-u)\mathscr{S}+\frac{3}{2}\big(\overline{H}_{12}+\overline{H}_{23}\big)+\frac{1}{2}\big(\overline{E}_{1}+\overline{E}_{3}\big) &\quad \text{if } 0\leqslant u\leqslant 1, \\
(2-u)\mathscr{S}+\frac{5-u}{2}\big(\overline{H}_{12}+\overline{H}_{23}\big)+\frac{1}{2}\big(\overline{E}_{1}+\overline{E}_{3}\big) &\quad \text{if } 1\leqslant u\leqslant 2.
\end{cases}
$$
Integrating, leads to $S_{Y}(\overline{E}_2)=\frac{51}{56}$, so that $\beta_{\overline{X}}(\overline{E}_2)>0$.
\end{proof}
Let us now show that $\beta(\mathbf{F})>0$ when the centre of $\mathbf{F}$ on $X$ is small.

\begin{lemma}[{\cite[Lemma 4.1]{Denisova}}]
\label{lemma:3-12-EL}
Suppose that $Z$ is a curve in $E_L$.
Then $\beta(\mathbf{F})>0$.
\end{lemma}

\begin{proof}
Note that $\overline{E}_L\cong\mathbb{P}^1\times\mathbb{P}^1$.
Let $\mathbf{s}$ be a section of the natural projection $\overline{E}_L\to L$ such that $\mathbf{s}^2=0$,
and let $\mathbf{l}$ be a fibre of this projection. Then
$\overline{E}_L\vert_{\overline{E}_L}\sim -\mathbf{s}+\mathbf{l}$
and $\overline{H}\vert_{\overline{E}_L}\sim \mathbf{l}$.

Set $C_Q=\overline{Q}\cap\overline{E}_L$ and $C_R=\overline{R}\cap\overline{E}_L$.
Then $C_Q$ and $C_R$ are smooth irreducible $G$-invariant curves.
Furthermore, these are the only $G$-invariant irreducible curves in the surface $\overline{E}_L$.
Therefore, we conclude that either $\overline{Z}=C_Q$ or $\overline{Z}=C_R$.
Note that $C_Q\sim C_R\sim \mathbf{s}+\mathbf{l}$.

We let $\mathscr{S}=\overline{E}_L$, $\mathscr{C}=\overline{Z}$.
Then $S_{Y}(\mathscr{S})=\frac{37}{56}<1$, see the proof of Lemma~\ref{lemma:3-12-divisorial}.

Let us compute $S(W^\mathscr{S}_{\bullet,\bullet};\mathscr{C})$.
We recall from from the proof of Lemma~\ref{lemma:3-12-divisorial} that $\tau=\frac{3}{2}$.
Moreover, it also follows from the proof of Lemma~\ref{lemma:3-12-divisorial} that
$$
P(u)\big\vert_{\mathscr{S}}-v\mathscr{C}\sim_{\mathbb{R}}\begin{cases}
(1+u-v)\mathbf{s}+(3-u-v)\mathbf{l} &\quad \text{if } 0\leqslant u\leqslant 1, \\
(2-v)\mathbf{s}+(6-4u-v)\mathbf{l} &\quad \text{if } 1\leqslant u\leqslant \frac{3}{2},
\end{cases}
$$
and
$$
N(u)\big\vert_{\mathscr{S}}=\begin{cases}
0 &\quad \text{if } 0\leqslant u\leqslant 1, \\
(u-1)\big(C_Q+\mathbf{l}_{12}+\mathbf{l}_{23}\big) &\quad \text{if } 1\leqslant u\leqslant \frac{3}{2},
\end{cases}
$$
where $\mathbf{l}_{12}=\overline{H}_{12}\cap\overline{E}_L$ and $\mathbf{l}_{23}=\overline{H}_{23}\cap\overline{E}_L$.
Note that $\mathbf{l}_{12}$ and $\mathbf{l}_{13}$ are fibres of the natural projection $\overline{E}_L\to L$ over the points $L\cap H_{12}$ and $L\cap H_{13}$, respectively.

We have $P(u,v)\sim_{\mathbb{R}}P(u)\vert_{\mathscr{S}}-v\mathscr{C}$ and $N(u,v)=0$ for $v\in[0,t(u)]$, where
$$
t(u)=\begin{cases}
1+u& \quad \text{if } 0\leqslant u \leqslant 1,\\
6-4u & \quad \text{if } 1\leqslant u \leqslant \frac{3}{2}.
\end{cases}
$$
Thus, if $\mathscr{C}=C_Q$, then
\begin{multline*}
\quad\quad\quad\quad S(W^\mathscr{S}_{\bullet,\bullet};\mathscr{C})=
\frac{3}{28}\int\limits_{1}^{\frac{3}{2}}(u-1)(24-16u)dvdu+\\
+\frac{3}{28}\int\limits_{0}^1\int\limits_{0}^{1+u}2(3-u-v)(1+u-v)dvdu+\frac{3}{28}\int\limits_{1}^{\frac{3}{2}}\int\limits_{0}^{6-4u}2(2-v)(6-4u-v)dvdu=\frac{159}{224}.
\end{multline*}
Similarly, if $\mathscr{C}=C_R$, then $S(W^\mathscr{S}_{\bullet,\bullet};\mathscr{C})=\frac{151}{224}$.
So, it follows from \eqref{equation:Kento-curve} that $\beta(\mathbf{F})>0$.
\end{proof}

We now study $G$-invariant irreducible curves lying on $\overline{E}_2$.

\begin{lemma}
\label{Z not in E2}
Suppose that $\pi(Z)=L_{\infty}=\ell_2$; then $\beta(\mathbf{F})>0$.
\end{lemma}

\begin{proof}
Observe that $\overline{Z}\subset\overline{E}_2$.
We will see soon that $\overline{Z}$ is a smooth $G$-irreducible curve,
so we let $\mathscr{S}=\overline{E}_2$, $\mathscr{S}=\overline{Z}$.
Then $S_Y(\mathscr{S})=\frac{51}{56}$ (see the proof of Lemma~\ref{lemma:3-12-divisorial}).

As in Lemma~\ref{lemma:3-12-points}, let $S=\{x_0x_2+x_1x_3=0\}$,
and let $\overline{S}$ be its proper transform on $\overline{X}$.
Set $C_S=\overline{S}\vert_{\mathscr{S}}$ and $C_R=\overline{R}\vert_{\mathscr{S}}$.
Using the map $[x_0:x_1:x_2:x_3]\mapsto [x_1:x_2]$, we can $G$-equivariantly identify
$\delta(\mathscr{S})=\mathbb{P}^1\times\mathbb{P}^1$ with coordinates $([x_0:x_3],[x_1:x_2])$ such that
the involution $\tau$ acts as $([x_0:x_3],[x_1:x_2])\mapsto([x_3:x_0],[x_2:x_1])$,
and $\Gamma\cong\mathbb{C}^\ast$ acts as
$$
\big([x_0:x_3],[x_1:x_2]\big)\mapsto\big([\lambda x_0:x_3],[\lambda x_1:x_2]\big),
$$
where $\lambda\in\mathbb{C}^\ast$. Then the only $G$-invariant irreducible curves in the surface $\delta(\mathscr{S})$ are the curves
$\delta(C_S)=\{x_0x_2+x_1x_3=0\}$ and $\delta(C_R)=\{x_0x_2-x_1x_3=0\}$,
which implies that $C_S$ and $C_R$ are the only $G$-invariant irreducible curves in $\mathscr{S}$,
so that $\mathscr{C}=C_S$ or $\mathscr{C}=C_R$.

The morphism $\delta$ in \eqref{equation:3-12-diagram} induces a
$G$-equivariant birational morphism $\sigma\colon\mathscr{S}\to\delta(\mathscr{S})$
that blows up the points $([0:1],[1:0])$ and $([1:0],[0:1])$,
which are not contained in the curves $\delta(C_S)$ and $\delta(C_R)$.
In particular, we see that $\mathscr{S}$ is a sextic del Pezzo surface.

Set $\mathbf{e}_1=\overline{E}_1\vert_{\mathscr{S}}$ and $\mathbf{e}_3=\overline{E}_3\vert_{\mathscr{S}}$.
Then $\mathbf{e}_1$ and $\mathbf{e}_3$ are the $\sigma$-exceptional curves
such that $\sigma(\mathbf{e}_1)=([0:1],[1:0])$ and $\sigma(\mathbf{e}_3)=([1:0],[0:1])$.
Let $\mathbf{s}_1$ and $\mathbf{s}_3$ be the proper transforms on $\mathscr{S}$ of the curves $\{x_2=0\}$ and $\{x_1=0\}$,
and~$\mathbf{l}_1$ and $\mathbf{l}_3$ be the proper transforms of the curves $\{x_2=0\}$ and $\{x_1=0\}$, respectively.
Then $\sigma(\mathbf{s}_1)$ and $\sigma(\mathbf{s}_3)$ are the sections of the natural projection $\delta(\mathscr{S})\to \ell_2$
that pass through the points $\delta(\mathbf{e}_1)$ and $\delta(\mathbf{e}_3)$, respectively,
and $\sigma(\mathbf{l}_1)$ and $\sigma(\mathbf{l}_3)$ are the fibres of this projection
that pass through the points $\delta(\mathbf{e}_1)$ and $\delta(\mathbf{e}_3)$, respectively.
Then $C_S\sim C_R\sim \mathbf{s}_1+\mathbf{l}_1+2\mathbf{e}_1\sim \mathbf{s}_3+\mathbf{l}_3+2\mathbf{e}_3$.

Recall that $\mathbf{e}_1$, $\mathbf{e}_3$, $\mathbf{s}_1$, $\mathbf{s}_3$, $\mathbf{l}_1$, $\mathbf{l}_3$
are all $(-1)$-curves in $\mathscr{S}$, they generate the Mori cone $\overline{\mathrm{NE}(\mathscr{S})}$.
Note that $\overline{H}\vert_{\mathscr{S}}\sim \mathbf{l}_1+\mathbf{e}_1$ and $\overline{E}_2\vert_{\mathscr{S}}\sim -\mathbf{s}_1+\mathbf{l}_1$,
and we have $\overline{H}_{12}\vert_{\mathscr{S}}=\mathbf{s}_1$ and $\overline{H}_{12}\vert_{\mathscr{S}}=\mathbf{s}_2$.
So, using the description of $P(u)$ and $N(u)$ obtained in the proof of Lemma~\ref{lemma:3-12-divisorial},
we get
$$
P(u)\big\vert_{\mathscr{S}}\sim_{\mathbb{R}}
\begin{cases}
3\mathbf{e}_1-\mathbf{e}_3+(3-u)\mathbf{l}_1+(u+1)\mathbf{s}_1 &\quad \text{if } 0\leqslant u\leqslant 1,\\
(4-u)\mathbf{e}_1+(u-2)\mathbf{e}_3+(3-u)\mathbf{l}_1+(3-u)\mathbf{s}_1 &\quad \text{if } 1\leqslant u\leqslant 2,
\end{cases}
$$
and
$$
N(u)\big\vert_{\mathscr{S}}=
\begin{cases}
0  & \quad \text{if } 0\leqslant u\leqslant 1,\\
(u-1)(\mathbf{s}_1+\mathbf{s}_2) & \quad \text{if } 1\leqslant u\leqslant 2.
\end{cases}
$$
In particular, we see that $\mathscr{C}\not\subset\mathrm{Supp}(N(u)\vert_{\mathscr{S}})$ for every $u\in[0,2]$.

Now, intersecting   $P(u)\vert_{\mathscr{S}}-v\mathscr{C}$ with $\mathbf{e}_1$, $\mathbf{e}_3$, $\mathbf{s}_1$, $\mathbf{s}_3$, $\mathbf{l}_1$, $\mathbf{l}_3$,
we find $P(u,v)$ and $N(u,v)$ for $u\in[0,2]$ and $v\in[0,t(u)]$.
If $u\in[0,1]$, then~$t(u)=1$,
$$
P(u,v)\sim_{\mathbb{R}}
\begin{cases}
(3-2v)\mathbf{e}_1-\mathbf{e}_3+(3-u-v)\mathbf{l}_1+(u-v+1)\mathbf{s}_1 &\quad\text{ if } 0\leqslant v \leqslant u,\\
(3-2v)\mathbf{e}_1-\mathbf{e}_3+(3-2v)\mathbf{l}_1+(u-v+1)\mathbf{s}_1 &\quad\text{ if } u\leqslant v\leqslant 1,
\end{cases}
$$
and
$$
N(u,v)=\begin{cases}
0 &\quad\text{ if } 0\leqslant v \leqslant u,\\
(v-u)\big(\mathbf{l}_1+\mathbf{l}_3\big) &\quad\text{ if } u\leqslant v\leqslant 1.
\end{cases}
$$
If $u\in[1,2]$, then
$t(u)=\frac{1}{2}$, $N(u,v)=0$ and $P(u,v)\sim_{\mathbb{R}} P(u)\vert_{\mathscr{S}}-v\mathscr{C}$.
This gives
$$
\big(P(u,v)\big)^2\sim_{\mathbb{R}}
\begin{cases}
4-2u^2+2v^2+4u-8v &\quad\text{ if } u\in[0,1], v\in [0,u],\\
4(1-v)(1+u-v) &\quad\text{ if } u\in[0,1], v\in [u,1],\\
2(1-2v)(5-2u-2v) &\quad\text{ if } u\in[1,2], v\in [0,0.5].
\end{cases}
$$
Now, by integrating we get $S(W^\mathscr{S}_{\bullet,\bullet};\mathscr{C})=\frac{9}{28}$,
so that  $\beta(\mathbf{F})>0$ by \eqref{equation:Kento-curve}.
\end{proof}

For $a\in\mathbb{C}^\ast$, let $\Pi_a=\{x_0-ax_1=x_3-ax_2\}\subset\mathbb{P}^3$.
Then $L_a\subset\Pi_a$, the plane $\Pi_a$ does not contain $L$, $\ell_1$, $\ell_2$, $\ell_3$, and neither $\ell_1\cap\ell_2$ nor $\ell_2\cap\ell_3$ lie on $\Pi_a$.
Set $P_1=\Pi_a\cap \ell_1$, $P_2=\Pi_a\cap \ell_2$, $P_3=\Pi_a\cap \ell_3$, $P_4=\Pi_a\cap L$.
Let $\overline{\Pi}_a$ be the preimage on  $\overline{X}$ of the plane $\Pi_a$.
Then $\varphi\circ\gamma$ in \eqref{equation:3-12-diagram} induces a~birational morphism $\overline{\Pi}_a\to\Pi_a$
that is a blowup of $P_1$, $P_2$, $P_3$, $P_4$.

\begin{lemma}
\label{lemma:3-12-dP5}
If $a\not\in\{1,2\}$, no three of $P_1, P_2, P_3$, and $P_4$ are collinear, and none of them lies on $L_a$.

When $a=1$, no three of $P_1, P_2, P_3$, and $P_4$ are collinear, $P_1$ and  $P_3$ lie on $L_1$, but $P_2$ and $P_4$ do not.
When $a=2$, then $P_1,P_3$ and $P_4$ lie on the line
$\Pi_2\cap\{x_0+x_1-x_2=0\}$, but $P_2$ does not, and none of $P_1,P_2, P_3$ or $P_4$ lies on $L_2$.
\end{lemma}

\begin{proof}
Left to the reader.
\end{proof}

Thus, if $a\ne 2$, $\overline{\Pi}_a$ is quintic del Pezzo surface, while if $a=2$, $\overline{\Pi}_a$ is a weak quintic del Pezzo surface.
In both cases, we let $Y=\overline{X}$ and $\mathscr{S}=\overline{\Pi}_a$.
Then
$$
-K_{Y}-u\mathscr{S}\sim_{\mathbb{R}}
\left(\frac{3}{2}-u\right)\mathscr{S}+
\frac{1}{2}\big(\overline{Q}+\overline{H}_{12}+\overline{H}_{23}\big)+\frac{1}{2}\overline{H}_L.
$$
Therefore, it follows from \eqref{equation:3-12-P1-P2} that $\tau=\frac{3}{2}$.
Moreover, if $0\leqslant u\leqslant 1$, then $N(u)=0$.
Furthermore, if $1\leqslant u\leqslant\frac{3}{2}$, then $N(u)=(u-1)(\overline{Q}+\overline{H}_{12}+\overline{H}_{23})$.
Thus, we have
$$
P(u)\sim_{\mathbb{R}}\begin{cases}
\left(\frac{3}{2}-u\right)\mathscr{S}+\frac{1}{2}\big(\overline{Q}+\overline{H}_{12}+\overline{H}_{23}\big)+\frac{1}{2}\overline{H}_L& \quad \text{if } 0\leqslant u\leqslant 1,\\
\left(\frac{3}{2}-u\right)\big(\mathscr{S}+\overline{Q}+\overline{H}_{12}+\overline{H}_{23}\big)+\frac{1}{2}\overline{H}_L& \quad \text{if } 1\leqslant u\leqslant \frac{3}{2},
\end{cases}
$$
By integrating we obtain $S_Y(\mathscr{S})=\frac{227}{448}$.

Now, let $\mathbf{e}_1$, $\mathbf{e}_2$, $\mathbf{e}_3$, $\mathbf{e}_4$
be exceptional curves of the blowup $\mathscr{S}\to\Pi_a$
that are mapped to the points $P_1$, $P_2$, $P_3$, $P_4$, respectively.
Then $\overline{E}_1\vert_{\mathscr{S}}=\mathbf{e}_1$, $\overline{E}_2\vert_{\mathscr{S}}=\mathbf{e}_2$, $\overline{E}_3\vert_{\mathscr{S}}=\mathbf{e}_3$, $\overline{E}_L\vert_{\mathscr{S}}=\mathbf{e}_4$.
Set $\mathbf{h}=\overline{H}\vert_{\mathscr{S}}$. Then
$$
P(u)\big\vert_{\mathscr{S}}\sim_{\mathbb{R}}\begin{cases}
(4-u)\mathbf{h}-\mathbf{e}_1-\mathbf{e}_2-\mathbf{e}_3-\mathbf{e}_4& \quad \text{if } 0\leqslant u\leqslant 1,\\
(8-5u)\mathbf{h}-(3-2u)\big(\mathbf{e}_1+\mathbf{e}_2+\mathbf{e}_3\big)-(2-u)\mathbf{e}_4& \quad \text{if } 1\leqslant u\leqslant \frac{3}{2}.
\end{cases}
$$
Note that $L_a\not\subset H_{12}\cup H_{23}$ for every $a\in\mathbb{C}^\ast$.
Moreover, one has $L_a\subset Q$ if and only if $a=2$.

\begin{lemma}
\label{lemma:3-12-La}
Suppose that $\pi(Z)=L_a$ for $a\in\mathbb{C}\setminus\{0,1,2\}$.
Then $\beta(\mathbf{F})>0$.
\end{lemma}

\begin{proof}
Let $\mathscr{C}=\overline{Z}$. Note that $\mathscr{C}\sim\mathbf{h}$.
Arguing as in the proof of \cite[Lemma 4.1]{Denisova}, we get
$$
t(u)=\begin{cases}
2-u& \quad \text{if } 0\leqslant u \leqslant 1,\\
\frac{5-3u}{2} & \quad \text{if } 1\leqslant u \leqslant \frac{7}{5},\\
6-4u & \quad \text{if } \frac{7}{5}\leqslant u \leqslant \frac{3}{2}.
\end{cases}
$$
Moreover, if $0\leqslant u\leqslant 1$, then
$P(u,v)\sim_{\mathbb{R}}(4-u-v)\mathbf{h}-\mathbf{e}_1-\mathbf{e}_2-\mathbf{e}_3-\mathbf{e}_4$ for $v\in[0,2-u]$.
Similarly, if $1\leqslant u \leqslant \frac{3}{2}$ and $0\leqslant v \leqslant 3-2u$, then
$$
P(u,v)\sim_{\mathbb{R}}(8-5u-v)\mathbf{h}-(3-2u)\big(\mathbf{e}_1+\mathbf{e}_2+\mathbf{e}_3\big)-(u-2)\mathbf{e}_4.
$$
Finally, if $1\leqslant u \leqslant \frac{3}{2}$ and $3-2u\leqslant v\leqslant t(u)$, then
$$
P(u,v)\sim_{\mathbb{R}}(17-11u-4v)\mathbf{h}-(6-4u-v)\big(\mathbf{e}_1+\mathbf{e}_2+\mathbf{e}_3\big)-(11-7u-3v)\mathbf{e}_4.
$$
This gives $S(W^\mathscr{S}_{\bullet,\bullet};\mathscr{C})=\frac{753}{1120}$, so that $\beta(\mathbf{F})>0$ by \eqref{equation:Kento-curve}, since $S_Y(\mathscr{S})=\frac{227}{448}$.
\end{proof}

\begin{lemma}
\label{lemma:3-12-L1}
Suppose that $\pi(Z)=L_1$.
Then $\beta(\mathbf{F})>0$.
\end{lemma}

\begin{proof}
Let $\mathscr{C}=\overline{Z}$. Then $\mathscr{C}\sim\mathbf{h}-\mathbf{e}_1-\mathbf{e}_3$.
Moreover, if $0\leqslant u \leqslant 1$, then $t(u)=3-u$.
Similarly, if $1\leqslant u \leqslant \frac{3}{2}$, then $t(u)=6-4u$.
Furthermore, if $0\leqslant u\leqslant 1$, then
$$
P(u,v)\sim_{\mathbb{R}}
\begin{cases}
(4-u-v)\mathbf{h}+(v-1)\big(\mathbf{e}_1+\mathbf{e}_3\big)-\mathbf{e}_2-\mathbf{e}_4 &\text{ if } 0\leqslant v \leqslant 1,\\
(4-u-v)\mathbf{h}-\mathbf{e}_2-\mathbf{e}_4 &\text{ if } 1\leqslant v \leqslant 2-u,\\
(3-2u-2v)\big(2\mathbf{h}-\mathbf{e}_2-\mathbf{e}_4\big) &\text{ if } 2-u\leqslant v \leqslant 3-u.
\end{cases}
$$
Similarly, if $1\leqslant u \leqslant \frac{3}{2}$ and $0\leqslant v \leqslant 3-2u$, then
$$
P(u,v)\sim_{\mathbb{R}}(8-5u-v)\mathbf{h}-(3-2u-v)\big(\mathbf{e}_1+\mathbf{e}_3\big)-(3-2u)\mathbf{e}_2-(u-2)\mathbf{e}_4.
$$
Finally, if $1\leqslant u \leqslant \frac{3}{2}$ and $3-2u\leqslant v\leqslant 6-4u$, then
$$
P(u,v)\sim_{\mathbb{R}}(11-7u-2v)\mathbf{h}-(6-4u-v)\mathbf{e}_2-(5-3u-v)\mathbf{e}_4.
$$
Therefore, we have
\begin{multline*}
S\big(W^\mathscr{S}_{\bullet,\bullet};\mathscr{C}\big)=\frac{3}{28}\int\limits_{0}^{1}\int\limits_{0}^1u^2+2uv-v^2-8u-4v+12dvdu+\\
+\frac{3}{28}\int\limits_{1}^{2-u}\int\limits_{1}^{2-u}u^2+2uv+v^2-8u-8v+14dvdu+\frac{3}{28}\int\limits_{2-u}^{3-u}\int\limits_{2-u}^{3-u}2(3-u-v)^2dvdu+\\
+\frac{3}{28}\int\limits_{1}^{\frac{3}{2}}\int\limits_{0}^{3-2u}12u^2+2uv-v^2-40u-4v+33dvdu+\frac{3}{28}\int\limits_{1}^{\frac{3}{2}}\int\limits_{3-2u}^{6-4u}2(6-4u-v)(5-3u-v)dvdu=\frac{31}{32}.
\end{multline*}
Thus, it follows from \eqref{equation:Kento-curve} that $\beta(\mathbf{F})>0$, because $S_Y(\mathscr{S})=\frac{227}{448}$.
\end{proof}

\begin{lemma}
\label{lemma:3-12-L2}
Suppose that $\pi(Z)=L_2$. Then $\beta(\mathbf{F})>0$.
\end{lemma}

\begin{proof}
Let us introduce several curves on the surface $\mathscr{S}=\overline{\Pi}_2$ as follows:
\begin{itemize}
\item let $\mathbf{l}_{12}$ be the proper transform of the line in $\Pi_2$ that contains $P_1$ and $P_2$,
\item let $\mathbf{l}_{23}$ be the proper transform of the line in $\Pi_2$ that contains $P_2$ and $P_3$,
\item let $\mathbf{l}_{24}$ be the proper transform of the line in $\Pi_2$ that contains $P_2$ and $P_4$,
\item let $\mathbf{l}$ be the proper transform on $\mathscr{S}$ of the line $\Pi_2\cap\{x_0+x_1-x_2=0\}$.
\end{itemize}
On $\mathscr{S}$, we have $\mathbf{l}_{12}\sim\mathbf{h}-\mathbf{e}_{1}-\mathbf{e}_{2}$, $\mathbf{l}_{23}\sim\mathbf{h}-\mathbf{e}_{2}-\mathbf{e}_{3}$.
$\mathbf{l}_{24}\sim\mathbf{h}-\mathbf{e}_{2}-\mathbf{e}_{4}$, $\mathbf{l}\sim\mathbf{h}-\mathbf{e}_{1}-\mathbf{e}_{3}-\mathbf{e}_{4}$.
Note that
$\mathbf{e}_1$, $\mathbf{e}_2$, $\mathbf{e}_3$, $\mathbf{e}_4$, $\mathbf{l}_{12}$, $\mathbf{l}_{23}$, $\mathbf{l}_{24}$ are all $(-1)$-curves in $\mathscr{S}$,
and $\mathbf{l}$ is the unique $(-2)$-curve in the surface $\mathscr{S}$.
By \cite[Proposition~8.5]{Coray1988},
these curves generate the Mori cone $\overline{\mathrm{NE}(\mathscr{S})}$.

Now, let $\mathscr{C}=\overline{Z}$. Then $\mathscr{C}\sim\mathbf{h}$.
Moreover, intersecting the divisors under consideration with the curves $\mathbf{e}_1$, $\mathbf{e}_2$, $\mathbf{e}_3$, $\mathbf{e}_4$, $\mathbf{l}_{12}$, $\mathbf{l}_{23}$, $\mathbf{l}_{24}$, $\mathbf{l}$,
we see that
$$
t(u)=\begin{cases}
\frac{7-3u}{3}& \quad \text{if } 0\leqslant u \leqslant 1,\\
\frac{10-6u}{3} & \quad \text{if } 1\leqslant u \leqslant \frac{4}{3},\\
6-4u & \quad \text{if } \frac{4}{3}\leqslant u \leqslant \frac{3}{2}.
\end{cases}
$$
Furthermore, if $0\leqslant u\leqslant 1$, then
$$
P(u,v)\sim_{\mathbb{R}}
\begin{cases}
(4-u-v)\mathbf{h}-\mathbf{e}_1-\mathbf{e}_2-\mathbf{e}_3-\mathbf{e}_4 &\text{ if } 0\leqslant v \leqslant 1-u,\\
\frac{3-u-v}{2}\big(3\mathbf{h}-\mathbf{e}_1-\mathbf{e}_3-\mathbf{e}_4\big)-\mathbf{e}_2 &\text{ if } 1-u\leqslant v \leqslant 2-u,\\
\frac{7-3u-3v}{2}\big(3\mathbf{h}-\mathbf{e}_1-2\mathbf{e}_2-\mathbf{e}_3-\mathbf{e}_4\big) &\text{ if } 2-u\leqslant v \leqslant \frac{7-3u}{3},
\end{cases}
$$
and
$$
N(u,v)=
\begin{cases}
0 &\quad\text{ if } 0\leqslant v \leqslant 1-u,\\
\frac{v+u-1}{2}\mathbf{l} &\quad\text{ if } 1-u\leqslant v \leqslant 2-u,\\
\frac{v+u-1}{2}\mathbf{l}+(u+v-2)\big(\mathbf{l}_{12}+\mathbf{l}_{23}+\mathbf{l}_{24}\big) &\quad\text{ if } 2-u\leqslant v \leqslant \frac{7-3u}{3}.
\end{cases}
$$
Similarly, if $1\leqslant u \leqslant \frac{4}{3}$, then $P(u,v)$ is $\mathbb{R}$-rationally equivalent to
$$
\begin{cases}
\frac{16-10u-3v}{2}\mathbf{h}-\frac{6-4u-v}{2}\big(\mathbf{e}_1+\mathbf{e}_3\big)-(3-2u)\mathbf{e}_2-\frac{4-2u-v}{2}\mathbf{e}_4 &\text{ if } 0\leqslant v \leqslant 3-2u,\\
\frac{22-14u-5v}{2}\mathbf{h}-\frac{6-4u-v}{2}\big(\mathbf{e}_1+2\mathbf{e}_2+\mathbf{e}_3\big)-\frac{10-6u-3v}{2}\mathbf{e}_4 &\text{ if } 3-2u\leqslant v \leqslant 2-u,\\
\frac{10-6u-3v}{2}\big(\mathbf{h}-\mathbf{e}_1+2\mathbf{e}_2-\mathbf{e}_3+\mathbf{e}_4\big) &\text{ if } 2-u\leqslant v\leqslant\frac{10-6u}{3},
\end{cases}
$$
and
$$
N(u,v)=
\begin{cases}
\frac{v}{2}\mathbf{l} &\quad\text{ if } 0\leqslant v \leqslant 3-2u,\\
\frac{v}{2}\mathbf{l}+(v+2u-3)\mathbf{l}_{24}&\quad\text{ if } 3-2u\leqslant v \leqslant 2-u,\\
\frac{v}{2}\mathbf{l}+(v+2u-3)\mathbf{l}_{24}+(u+v-2)\big(\mathbf{l}_{12}+\mathbf{l}_{23}\big)&\quad\text{ if } 2-u\leqslant v\leqslant\frac{10-6u}{3}.
\end{cases}
$$
Likewise, if $\frac{4}{3}\leqslant u \leqslant \frac{3}{2}$ and $0\leqslant v \leqslant 3-2u$, then
$$
P(u,v)\sim_{\mathbb{R}}\frac{16-10u-3v}{2}\mathbf{h}-\frac{6-4u-v}{2}\big(\mathbf{e}_1+\mathbf{e}_3\big)-(3-2u)\mathbf{e}_2-\frac{4-2u-v}{2}\mathbf{e}_4
$$
and $N(u,v)=\frac{v}{2}\mathbf{l}$. Finally, if $\frac{4}{3}\leqslant u \leqslant \frac{3}{2}$ and $3-2u\leqslant v\leqslant 6-4u$, then
$$
P(u,v)\sim_{\mathbb{R}}\frac{22-14u-5v}{2}\mathbf{h}-\frac{6-4u-v}{2}\big(\mathbf{e}_1+2\mathbf{e}_2+\mathbf{e}_3\big)-\frac{10-6u-3v}{2}\mathbf{e}_4
$$
and $N(u,v)=\frac{v}{2}\mathbf{l}+(v+2u-3)\mathbf{l}_{24}$.

If $1\leqslant u\leqslant\frac{3}{2}$, then $\mathscr{C}\subset\mathrm{Supp}(N(u))$ and $\mathrm{ord}_{\mathscr{C}}(N(u)\vert_{\mathscr{S}})=(u-1)$.
Thus, we have
\begin{multline*}
\quad \quad \quad \quad S\big(W^\mathscr{S}_{\bullet,\bullet};\mathscr{C}\big)=\frac{3}{28}\int\limits_{1}^{\frac{3}{2}}(12u^2-40u+33)(u-1)du+\\
+\frac{3}{28}\int\limits_{0}^{1}\int\limits_{0}^{1-u}u^2+2uv+v^2-8u-8v+12dvdu+\frac{3}{28}\int\limits_{0}^{1}\int\limits_{1-u}^{2-u}\frac{3u^2+6uv+3v^2-18u-18v+25}{2}dvdu+\\
+\frac{3}{28}\int\limits_{0}^{1}\int\limits_{2-u}^{\frac{7-3u}{3}}\frac{(7-3u-3v)^2}{2}dvdu+\frac{3}{28}\int\limits_{1}^{\frac{4}{3}}\int\limits_{0}^{3-2u}\frac{24u^2+20uv+3v^2-80u-32v+66}{2}dvdu+\\
+\frac{3}{28}\int\limits_{1}^{\frac{4}{3}}\int\limits_{3-2u}^{2-u}\frac{(14-8u-5v)(6-4u-v)}{2}dvdu+\frac{3}{28}\int\limits_{1}^{\frac{4}{3}}\int\limits_{2-u}^{\frac{10-6u}{3}}\frac{(10-6u-3v)^2}{2}dvdu+\\
+\frac{3}{28}\int\limits_{\frac{4}{3}}^{\frac{3}{2}}\int\limits_{0}^{3-2u}\frac{24u^2+20uv+3v^2-80u-32v+66}{2}dvdu+\frac{3}{28}\int\limits_{\frac{4}{3}}^{\frac{3}{2}}\int\limits_{3-2u}^{6-4u}\frac{(14-8u-5v)(6-4u-v)}{2}dvdu.
\end{multline*}
This gives $S(W^\mathscr{S}_{\bullet,\bullet};\mathscr{C})=\frac{2885}{4032}<1$.
Then $\beta(\mathbf{F})>0$ by \eqref{equation:Kento-curve}, since $S_Y(\mathscr{S})=\frac{227}{448}<1$.
\end{proof}

This finishes the proof that $X$ is K-polystable.

\section{K-polystability of the Fano 3-fold $X_{\infty}$ in Family \ref{example:4-13}}
\label{section:4-13}

Let $X=X_{\infty}$ be the 3-fold described in Family \ref{example:4-13}, and use the notation introduced in Section~\ref{section:Abban-Zhuang}.
Then $\pi\colon X\to (\mathbb{P}^1)^3$ is the blowup of
$$
C_\infty=\big\{x_0y_1-x_1y_0=0, x_0x_1(x_0z_1-x_1z_0)=0\big\}\subset\big(\mathbb{P}^1\big)^3,
$$
where $([x_0:x_1],[y_0:y_1],[z_0:z_1])$ are coordinates on $(\mathbb{P}^1)^3$.
Note that $C_{\infty}=C+L_1+L_2$, where
$C=\{x_0y_1-x_1y_0=0,x_0z_1-x_1z_0=0\}$, $L_1=\{x_0=0,y_0=0\}$, $L_2=\{x_1=0,y_1=0\}$.

\paragraph*{\bf Description of the automorphism group}
Let us consider the automorphisms $\iota_1$, $\iota_2$, $\tau_\lambda$ of $(\mathbb{P}^1)^3$ defined by
\begin{align*}
\iota_1\colon \big([x_0:x_1],[y_0:y_1],[z_0:z_1]\big)&\mapsto\big([x_1:x_0],[y_1:y_0],[z_1:z_0]\big),\\
\iota_2\colon\big([x_0:x_1],[y_0:y_1],[z_0:z_1]\big)&\mapsto\big([y_0:y_1],[x_0:x_1],[z_0:z_1]\big),\\
\tau_{\lambda}\colon \big([x_0:x_1],[y_0:y_1],[z_0:z_1]\big)&\mapsto\big([\lambda x_0:x_1],[\lambda y_0:y_1],[\lambda z_0:z_1]),
\end{align*}
where $\lambda\in\mathbb{C}^\ast$. As $\iota_1, \iota_2$ and $\tau_\lambda$ preserve $C_{\infty}$, their actions lift to  $X$,
and we can consider them as automorphisms of  $X$.
The group $G=\mathrm{Aut}(X)$ is generated by these automorphisms, so that
$\mathrm{Aut}(X)\cong(\mathbb{C}^\ast\rtimes\mumu_2)\times\mumu_2$,
because $\iota_2$ lies in the centre of $\mathrm{Aut}(X)$.

We will view $G$ as an infinite subgroup in $\mathrm{Aut}(\mathbb{P}^1\times\mathbb{P}^1\times\mathbb{P}^1)$
generated by automorphisms $\iota_1$, $\iota_2$ and $\tau_\lambda$ for $\lambda\in\mathbb{C}^\ast$.

\paragraph*{\bf Description of the $G$-invariant loci}

Set
$C^\prime=\{x_0y_1-x_1y_0=0, x_0z_1+x_1z_0=0\}\subset(\mathbb{P}^1)^3$.

\begin{lemma}
\label{lemma:4-13-curves} There is no point of $\mathbb{P}^1\times\mathbb{P}^1\times\mathbb{P}^1$ fixed by $G$.
The only $G$-invariant irreducible curves in $\mathbb{P}^1\times\mathbb{P}^1\times\mathbb{P}^1$ are $C$ and $C^\prime$.
\end{lemma}

\begin{proof}
See the proof of Lemma~\ref{lemma:2-22-curves} or the proof of \cite[Lemma 5.112]{Book}.
\end{proof}

Let $R_{x,y}=\{x_0y_1-x_1y_0=0\}$, $R_{x,z}=\{x_0z_1-x_1z_0=0\}$ and $R_{y,z}=\{y_0z_1-y_1z_0=0\}$.
Then $R_{x,y}$ and $R_{x,z}+R_{y,z}$ are $G$-invariant and $G$-irreducible,
and $C=R_{x,y}\cap R_{x,z}\cap R_{y,z}$.
Let $\varphi\colon W\to (\mathbb{P}^1)^3$ be the blowup of the curve $C$.
Then $W$ is a smooth Fano 3-fold in the Mori-Mukai Family \textnumero 4.6,
and fits into a  $G$-equivariant commutative diagram:
\begin{equation}
\label{equation:4-13-diagram}
\xymatrix{
&W\ar[dl]_\varphi\ar[dr]\ar[dr]^\varpi&\\
\mathbb{P}^1\times\mathbb{P}^1\times\mathbb{P}^1\ar@{-->}[rr]_{\chi}&&\mathbb{P}^3}
\end{equation}
where $\chi$ is a birational map that is given by
$$\hspace*{-0.3cm}
\big([x_0:x_1],[y_0:y_1],[z_0:z_1]\big)\mapsto\big[x_0y_0z_1-x_0y_1z_0:x_0y_0z_1-x_1y_0z_0:x_0y_1z_1-x_1y_0z_1:x_0y_1z_1-x_1y_1z_0\big]
$$
and $\varpi$ contracts the proper transforms of
$R_{x,y}$, $R_{x,z}$, $R_{y,z}$ to disjoint lines $\{x-y=0,z=0\}$, $\{y=0,t=0\}$, $\{x=0,z-t=0\}$, respectively,
where $[x:y:z:t]$ are coordinates on $\mathbb{P}^3$.

Let $\delta\colon \overline{X}\to W$ be the blowup of the proper transforms of the curves $L_1$ and $L_2$,
let~$\gamma\colon V\to (\mathbb{P}^1)^3$ be the blowup of  $L_1$ and $L_2$,
and let $\phi\colon\widetilde{X}\to V$ be the blowup of the proper transform of  $C$.
Then we have a $G$-equivariant commutative diagram
\begin{equation}
\label{equation:4-13-big-diagram}
\xymatrix{
\widetilde{X}\ar[rr]\ar[d]_{\phi} && X\ar[d]_{\pi} && \overline{X}\ar[ll]\ar[d]^{\delta}\\
V \ar[rr]_{\gamma} && \mathbb{P}^1\times\mathbb{P}^1\times\mathbb{P}^1 && W\ar[ll]^\varphi}
\end{equation}
where $\widetilde{X}\to X$ and $\overline{X}\to X$ are  $G$-equivariant small resolutions of the 3-fold~$X$.
Recall from Section~\ref{section:Abban-Zhuang} that either $Y=\widetilde{X}$ or $Y=\overline{X}$.

Let $E$, $F_1$, $F_2$ be the $\pi$-exceptional surfaces such that $\pi(E)=C$, $\pi(F_1)=L_1$,~\mbox{$\pi(F_2)=L_2$},
let $H_{x}$, $H_{y}$, $H_{z}$ be general fibres of the fibrations $\mathrm{pr}_x\circ\pi$,  $\mathrm{pr}_y\circ\pi$, $\mathrm{pr}_z\circ\pi$,
respectively, where $\mathrm{pr}_x$, $\mathrm{pr}_y$, $\mathrm{pr}_z$ are projections
of $(\mathbb{P}^1)^3$ to the first, second, third factor, respectively.
Let us denote by symbols $\overline{E}$, $\overline{F}_1$, $\overline{F}_2$, $\overline{R}_{x,y}$, $\overline{R}_{x,z}$, $\overline{R}_{y,z}$,
$\overline{H}_{x}$, $\overline{H}_{y}$, $\overline{H}_{z}$
the proper transforms of  ${E}$, ${F}_1$, ${F}_2$, ${R}_{x,y}$, ${R}_{x,z}$, ${R}_{y,z}$,
${H}_{x}$, ${H}_{y}$, ${H}_{z}$  on $\overline{X}$.
Then
$$
-K_{\overline{X}}\sim 2\overline{R}_{x,y}+\overline{E}+\overline{F}_1+\overline{F}_2+2\overline{H}_{z}
\sim \overline{R}_{x,y}+\overline{R}_{x,z}+\overline{R}_{y,z}+2\overline{E}.
$$
Note that the divisors
$\overline{E}$, $\overline{F}_1+\overline{F}_2$, $\overline{R}_{x,y}$, $\overline{R}_{x,z}+\overline{R}_{y,z}$ are $G$-invariant and $G$-irreducible,
but the pencils $|\overline{H}_{x}|$, $|\overline{H}_{y}|$, $|\overline{H}_{z}|$ do not contain $G$-invariant members.

We now prove that $X$ is K-polystable.
Suppose that  $X$ is not K-polystable.
By \cite[Corollary~4.14]{Zhuang}, there exists a $G$-invariant prime divisor $\mathbf{F}$ over $X$ such that
\begin{equation}
\label{equation:4-13-beta}
\beta(\mathbf{F})=A_{X}\big(\mathbf{F}\big)-S_X\big(\mathbf{F}\big)\leqslant 0.
\end{equation}
Let $Z$ and $\overline{Z}$ be its centres on $X$ and $\overline{X}$, respectively.
WE first consider the case when $Z$ is a divisor.

\begin{lemma}
\label{lemma:4-13-divisorial}
Let $S$ be a $G$-invariant prime divisor in $\overline{X}$. Then $\beta(S)>0$.
\end{lemma}

\begin{proof}
If $\beta(S)\leqslant 0$, then the divisor $-K_{\overline{S}}-S$ is big by \eqref{equation:KentoKyoto}.
On the other hand, arguing as in the proof of \cite[Lemma~5.113]{Book},
we see that $-K_{\overline{S}}-S$ is big only if $S=\overline{R}_{x,y}$ or $S=\overline{E}$.
Thus, to complete the proof, it is enough to show that $\beta(\overline{R}_{x,y})>0$ and $\beta(\overline{E})>0$.

Set $Y=\overline{X}$ and $\mathscr{S}=\overline{R}_{x,y}$.
Then $\tau=2$. Arguing as in the proof of \cite[Lemma 5.114]{Book}, we see that
$N(u)=0$ for $u\in[0,1]$, and $N(u)=(u-1)(\overline{E}+\overline{F}_1+\overline{F}_1)$ for $u\in[1,2]$.
Then
$$
P(u)\sim_{\mathbb{R}}\begin{cases}
(2-u)\overline{R}_{x,y}+\overline{E}+\overline{F}_1+\overline{F}_2+2\overline{H}_{z}  &\quad \text{if } 0\leqslant u\leqslant 1, \\
(2-u)\big(\overline{H}_x+\overline{H}_y\big)+2\overline{H}_z  &\quad \text{if } 1\leqslant u\leqslant 2.
\end{cases}
$$
Now, a direct computation shows that
$$
S_{Y}\big(\overline{R}_{x,y}\big)=\frac{1}{26}\int\limits_{0}^126-2u^3-6u^2-6udu+\frac{1}{26}\int\limits_{1}^{2}12(2-u)^2du=\frac{49}{52},
$$
so that  $\beta(\overline{R}_{x,y})=1-S_{Y}(\overline{R}_{x,y})=\frac{3}{52}>0$.

Now, we compute $\beta(\overline{E})$.
We set $Y=\overline{X}$ and $\mathscr{S}=\overline{E}$.
Then it follows from \eqref{equation:4-13-diagram} that the divisor $\overline{R}_{x,y}+\overline{R}_{x,z}+\overline{R}_{y,z}$ is not big,
which implies $\tau=2$. Moreover, we have
$$
P(u)\sim_{\mathbb{R}}\begin{cases}
(2-u)\overline{E}+\overline{R}_{x,y}+\overline{R}_{x,z}+\overline{R}_{y,z}&\quad \text{if } 0\leqslant u\leqslant 1, \\
(2-u)\big(\overline{E}+\overline{R}_{x,y}+\overline{R}_{x,z}+\overline{R}_{y,z}\big) &\quad \text{if } 1\leqslant u\leqslant 2,
\end{cases}
$$
and
$$
N(u)=\begin{cases}
   0  &\quad \text{if } 0\leqslant u\leqslant 1, \\
  (u-1)\big(\overline{R}_{x,y}+\overline{R}_{x,z}+\overline{R}_{y,z}\big) &\quad \text{if } 1\leqslant u\leqslant 2.
\end{cases}
$$
Now, we compute
$$
S_{Y}(\overline{E})=\frac{1}{26}\int\limits_{0}^14u^3-6u^2-18u+26du+\frac{1}{26}\int\limits_{1}^{2}6(2-u)^3du=\frac{35}{52},
$$
so that $\beta(\overline{E})=1-S_{Y}(\overline{E})=\frac{17}{52}$. This completes the proof of the lemma.
\end{proof}

We now consider the case when $Z$ and $\overline{Z}$  are small.  By Lemmas~\ref{lemma:4-13-curves} , we only need to consider curves, and
either $\pi(Z)=C$ and $\overline{Z}\subset\overline{E}$,
or $\pi(Z)=C^\prime$ and $\overline{Z}\subset\overline{R}_{x,y}$.

\begin{lemma}
\label{lemma:4-13-Z-in-E}
The curve $\overline{Z}$ is not contained in the surface $\overline{R}_{x,y}$.
\end{lemma}

\begin{proof}
The proof is very similar to the proof of \cite[Lemma 5.114]{Book}.
Suppose that $\overline{Z}\subset\overline{R}_{x,y}$.
Let $\overline{C}_{x,y}=\overline{E}\cap\overline{R}_{x,y}$,
and let $\overline{C}^\prime$ be the proper transform on the 3-fold $\overline{X}$ of the curve $C^\prime$.
Then it follows from Lemma~\ref{lemma:4-13-curves} that either $\overline{Z}=\overline{C}_{x,y}$ or $\overline{Z}=\overline{C}^\prime$.

Set $Y=\overline{X}$, $\mathscr{S}=\overline{R}_{x,y}$, $\mathscr{C}=\overline{Z}$.
Recall from the proof of Lemma~\ref{lemma:4-13-divisorial} that $S_{Y}(\mathscr{S})=\frac{49}{52}$.
Thus, it follows from \eqref{equation:Kento-curve} that $S(W^\mathscr{S}_{\bullet,\bullet};\mathscr{C})\geqslant 1$.
Let us compute $S(W^\mathscr{S}_{\bullet,\bullet};\mathscr{C})$.

Let $\ell_1$ and $\ell_2$ be the rulings of the surface $\mathscr{S}\cong\mathbb{P}^1\times\mathbb{P}^1$ such that $\mathrm{pr}_{x}\circ\pi$ and $\mathrm{pr}_{y}\circ\pi$ contract $\ell_1$, and $\mathrm{pr}_{z}\circ\pi$ contracts $\ell_2$.
Then it follows from the proof of Lemma~\ref{lemma:4-13-divisorial} that
$$
P(u)\big\vert_{\mathscr{S}}-v\mathscr{C}\sim_{\mathbb{R}}\begin{cases}
(1+u-v)\ell_1+(1+u-v)\ell_2  &\quad  0\leqslant u\leqslant 1, \\
(4-2u-3v)\ell_1+(2-v)\ell_2 &\quad  1\leqslant u\leqslant 2,
\end{cases}
$$
and
$$
N(u)\big\vert_{\mathscr{S}}=\begin{cases}
0    &\quad  0\leqslant u\leqslant 1, \\
(u-1)\big(\overline{C}_{x,y}+\overline{f}_1+\overline{f}_2\big)  &\quad  1\leqslant u\leqslant 2,
\end{cases}
$$
where $\overline{f}_1=\overline{F}_1\vert_{\mathscr{S}}$ and $\overline{f}_2=\overline{F}_2\vert_{\mathscr{S}}$.
Thus, if $0\leqslant u\leqslant 1$, then $t(u)=1+u$.
Similarly, if $1\leqslant u\leqslant 2$, then $t(u)=4-2u$.
We have $P(u,v)\sim_{\mathbb{R}} P(u)\big\vert_{\mathscr{S}}-v\mathscr{C}$ for every $v\in[0,t(u)]$.
Hence, if $\overline{Z}=\overline{C}_{x,y}$, then $S(W^\mathscr{S}_{\bullet,\bullet};\mathscr{C})$
can be computed as follows:
$$
\frac{3}{26}\int\limits_{1}^{2}(16-8u)(u-1)du+\frac{3}{26}\int\limits_{0}^{1}\int\limits_{0}^{1+u}2(1+u-v)^2dvdu+\frac{3}{26}\int\limits_{1}^{2}\int\limits_{0}^{4-2u}2(4-2u-v)(2-v)dvdu=\frac{35}{52}.
$$
Similarly, if $\overline{Z}=\overline{C}^\prime$, then $S(W^\mathscr{S}_{\bullet,\bullet};\mathscr{C})=\frac{27}{52}<1$.
This is a contradiction.
\end{proof}

Thus, we see that $\pi(Z)=C$, $\overline{Z}\subset\overline{E}$ and $\overline{Z}\ne\overline{E}\cap\overline{R}_{x,y}$.
Observe that $\delta\big(\overline{E}\big)\cong\mathbb{P}^1\times\mathbb{P}^1$.
As in the proof of Lemma~\ref{lemma:4-13-Z-in-E}, let $\overline{C}_{x,y}=\overline{E}\cap\overline{R}_{x,y}$.
Set $\overline{f}_1=\overline{F}_1\cap \overline{E}$ and $\overline{f}_2=\overline{F}_1\cap \overline{E}$.
Then~$\delta$ in \eqref{equation:4-13-big-diagram} induces~a~birational morphism $\overline{E}\to\delta(\overline{E})$ that contracts
$\overline{f}_1$ and $\overline{f}_2$ to two distinct points on the curve $\delta(\overline{C}_{x,y})$, which is a section the natural projection $\delta(\overline{E})\to C$.
Denote by~$\ell_1$ and $\ell_2$ the proper transforms on  $\overline{E}$ of the fibres of this projection that pass through the points $\delta(\overline{f}_1)$ and  $\delta(\overline{f}_2)$, respectively.
Note that $\overline{E}$ is a weak del Pezzo surface,
the curves $\overline{f}_1$, $\overline{f}_2$, $\ell_1$, $\ell_2$ are all $(-1)$-curves in $\overline{E}$,
and $\overline{C}_{x,y}$ is the only $(-2)$-curve in $\overline{E}$.
By \cite[Proposition~8.3]{Coray1988}, the curves $\overline{C}_{x,y}$, $\overline{f}_1$, $\overline{f}_2$, $\ell_1$, $\ell_2$ generate the Mori cone $\overline{\mathrm{NE}(\overline{E})}$.

Now, set $Y=\overline{X}$, $\mathscr{S}=\overline{E}$, $\mathscr{C}=\overline{Z}$.
Then it follows from the proof of Lemma~\ref{lemma:4-13-divisorial} that
$$
P(u)\big\vert_{\mathscr{S}}-v\mathscr{C}\sim_{\mathbb{R}}\begin{cases}
(1+u-v)\overline{C}+(2-v)\big(\overline{f}_1+\overline{f}_2\big)+(2-u)\big(\ell_1+\ell_2\big)&\quad \text{if } 0\leqslant u\leqslant 1, \\
(4-2u-v)\big(\overline{C}+\overline{f}_1+\overline{f}_2\big)+(2-u)\big(\ell_1+\ell_2\big)&\quad \text{if } 1\leqslant u\leqslant 2,
\end{cases}
$$
and
$$
N(u)\big\vert_{\mathscr{S}}-v\mathscr{C}=\begin{cases}
   0  &\quad \text{if } 0\leqslant u\leqslant 1, \\
  (u-1)\big(\overline{C}_{x,y}+\overline{C}_{x,z}+\overline{C}_{y,z}\big) &\quad \text{if } 1\leqslant u\leqslant 2,
\end{cases}
$$
where $\overline{C}_{x,z}=\overline{E}\cap\overline{R}_{x,z}$ and $\overline{C}_{y,z}=\overline{E}\cap\overline{R}_{y,z}$.
Thus, if $0\leqslant u\leqslant 1$, then $t(u)=1+u$.
Similarly, if $1\leqslant u\leqslant 2$, then $t(u)=4-2u$.
Moreover, if $0\leqslant u\leqslant 1$, then
$$
P(u,v)\sim_{\mathbb{R}}
\begin{cases}
(1+u-v)\overline{C}+(2-v)\big(\overline{f}_1+\overline{f}_2\big)+(2-u)\big(\ell_1+\ell_2\big)&\quad \text{ if } 0\leqslant v \leqslant u,\\
(1+u-v)\overline{C}+(2-v)\big(\overline{f}_1+\overline{f}_2+\ell_1+\ell_2\big)&\quad \text{ if } u\leqslant v \leqslant 1+u,
\end{cases}
$$
and
$$
N(u,v)=
\begin{cases}
0&\quad \text{ if } 0\leqslant v \leqslant u,\\
(v-u)\big(\ell_1+\ell_2\big) &\quad \text{ if } u\leqslant v \leqslant 1+u,
\end{cases}
$$
so that
$$
\big(P(u,v)\big)^2\sim_{\mathbb{R}}
\begin{cases}
7-3u^2+2uv+v^2+6u-10v&\quad \text{ if } 0\leqslant v \leqslant u,\\
(1+u-v)(7-u-3v) &\quad \text{ if } u\leqslant v \leqslant 1+u.
\end{cases}
$$
Furthermore, if $1\leqslant u\leqslant 2$, then
$$
P(u,v)\sim_{\mathbb{R}}
\begin{cases}
(4-2u-v)\big(\overline{C}+\overline{f}_1+\overline{f}_2\big)+(2-u)\big(\ell_1+\ell_2\big)&\quad \text{ if } 0\leqslant v \leqslant 2-u,\\
(4-2u-v)\big(\overline{C}+\overline{f}_1+\overline{f}_2+\ell_1+\ell_2\big) &\quad \text{ if } 2-u\leqslant v \leqslant 4-2u,
\end{cases}
$$
and
$$
N(u,v)=
\begin{cases}
0&\quad \text{ if } 0\leqslant v \leqslant 2-u,\\
(v+u-2)\big(\ell_1+\ell_2\big) &\quad \text{ if } 2-u\leqslant v \leqslant 4-2u,
\end{cases}
$$
so that
$$
\big(P(u,v)\big)^2\sim_{\mathbb{R}}
\begin{cases}
10u^2+8uv+v^2-40u-16v+40&\quad \text{ if } 0\leqslant v \leqslant 2-u,\\
3(4-2u-v)^2&\quad \text{ if } 2-u\leqslant v \leqslant 4-2u.
\end{cases}
$$
Note that $\overline{Z}\not\subset\mathrm{Supp}(N(u))$ for $u\in[0,2]$.
Thus, integrating, we get $S(W^\mathscr{S}_{\bullet,\bullet};\mathscr{C})=\frac{87}{104}$.
On the other hand, we already know from the proof of Lemma~\ref{lemma:4-13-divisorial} that $S_{Y}(\mathscr{S})=\frac{35}{52}<1$.
Then $\beta(\mathbf{F})>0$ by \eqref{equation:Kento-curve}, which contradicts \eqref{equation:4-13-beta}.
This shows that $X$ is K-polystable.

\section{Non-toric K-polystable singular  Fano 3-fold in Family \ref{example:3-13}}
\label{section:3-13}

All K-polystable smooth Fano 3-folds in the deformation family \textnumero 3.13 are described in Family \ref{example:3-13}.
This deformation family also contains some interesting members, which are worth mentioning. It contains a unique strictly K-semistable smooth member, whose automorphism group is isomorphic to $\mathbb{G}_a \rtimes \mathfrak{S}_3$ (see \cite[Lemma 5.98]{Book}). Recall that the automorphism group of a K-polystable Fano variety is reductive. The family contains a singular K-polystable toric Fano 3-fold which was already discussed in  Family \ref{example:3-13}. It also contains a non-toric complete intersection in $\mathbb{P}^2\times\mathbb{P}^2\times\mathbb{P}^2$ with one ordinary double point. In Example \ref{example:3-13-ODP-Ga} below we describe this Fano 3-fold. Its geometry deludes at first that it may be the missing K-polystable limit. However, this can be proven wrong as we discuss below.

\begin{example}
\label{example:3-13-ODP-Ga}
Let $Q$ be a smooth quadric 3-fold in $\mathbb{P}^4$, let $\ell_1$, $\ell_2$, $\ell_3$ be three general disjoint lines in $Q$, let $\zeta\colon Q\dasharrow \mathbb{P}^2\times\mathbb{P}^2\times\mathbb{P}^2$ be the rational map given by the product of the linear projections $Q\dasharrow \mathbb{P}^2$ from the lines $\ell_1$, $\ell_2$, $\ell_3$,
and let $X$ be the image of $\zeta$.
Then  $X$ is a singular Fano 3-fold in the deformation family \textnumero 3.13 that has one ordinary double point.
Namely, in suitable coordinates, $X\subset \mathbb P^2\times \mathbb P^2\times \mathbb P^2$ is the complete intersection
$$
\left\{\aligned
&x_0y_0+x_1y_1+x_2y_2-x_0y_2+x_1y_2+x_2y_0-x_2y_1=0,\\
&y_0z_0+y_1z_1+y_2z_2+y_0z_1-y_0z_2-y_1z_0+y_2z_0=0,\\
&x_0z_0+x_1z_1+x_2z_2-x_0z_1+x_1z_0-x_1z_2+x_2z_1=0,
\endaligned
\right.
$$
The ordinary double point is at $([1:0:0],[0:1:0],[0:0:1])$.
The morphism $\theta\colon Y\to Q$ is the blowup of the lines $\ell_1$, $\ell_2$, $\ell_3$, and resolves the indeterminacy of the rational map $\zeta$, fitting in
a commutative diagram
$$
\xymatrix{
&Y\ar@{->}[dl]_{\theta}\ar@{->}[dr]^{\pi}&\\
Q\ar@{-->}[rr]^{\zeta}&&X}.
$$
In the diagram, $\pi$ contracts the strict transform of the unique line in $Q$ that intersects the lines $\ell_1$, $\ell_2$, $\ell_3$ to the singular point of $X$.
Note that after flopping the $\pi$-exceptional curve, we obtain a similar commutative diagram.
Since $\mathrm{rk}\,\mathrm{Pic}(X)=3=\mathrm{rk}\,\mathrm{Cl}(X)-1$, $X$ is maximally non-factorial~\cite{CKMS}.
The automorphism group of $X$ is $\mathbb{G}_a\rtimes \mathfrak{S}_3$,
which is not reductive.
Hence, $X$ is not K-polystable. It can also be verified numerically: if  $\eta\colon V\to Y$ denotes the blowup of the flopping curve, and $E$ the $\eta$-exceptional surface, then $\beta(E)=-\frac{1}{40}$, so $X$ is K-unstable.
If the lines $\ell_1$, $\ell_2$, $\ell_3$ were contained in a hyperplane,
$X$ would have a curve of singularities, its automorphism group would be $\mathrm{PGL}_2(\mathbb{C})\rtimes \mathfrak{S}_3$, and it would still be K-unstable.
\end{example}

\paragraph{\bf The non-toric K-polystable limit}
Let us construct a singular non-toric K-polystable Fano 3-fold,
which admits a smoothing in the deformation family \textnumero 3.13.
Consider the smooth quadric 3-fold
$$
Q=\big\{xy+yz+xz+tw=0\big\}\subset\mathbb{P}^4,
$$
and the smooth conic $C=\big\{xy+yz+xz=t=w=0\big\}\subset Q$,
where $x$, $y$, $z$, $t$, $w$ are homogeneous coordinates on $\mathbb{P}^4$.

Now, denote by
$P_1=[1:0:0:0:0]$, $P_2=[0:1:0:0:0]$, and $P_3=[0:0:1:0:0]$, and observe these points lie on $C$.
Let $\theta\colon Y\to Q$ be the blowup of $P_1,P_2$ and $P_3$, $\mathcal{C}$ the proper transform on $Y$ of $C$,
and $\eta\colon V\to Y$ the blowup of $\mathcal C$.  Denote by $E$ the $\eta$-exceptional surface.
Then $V$ is the unique smooth Fano 3-fold in the deformation family \textnumero 5.1,
and $E\cong\mathbb{P}^1\times\mathbb{P}^1$ with $E\vert_{E}\sim\mathcal{O}_{E}(-1,-1)$.

This implies that $-K_Y$ is nef and big,
and $\mathcal{C}$ is the only curve in $Y$ that has trivial intersection with $-K_Y$.
Hence, for $m\in\mathbb{N}$ large enough, the linear system $|-mK_Y|$ gives a birational morphism $\pi\colon Y\to X$
that contracts  $\mathcal{C}$.
Then $X$ is a singular Fano $3$-fold with an ordinary double point at $\pi(\mathcal{C})$.

The $3$-fold $X$ can also be described as a subscheme of $\mathbb{P}^2\times\mathbb{P}^2\times\mathbb{P}^2$.
Indeed, consider the map $\zeta\colon Q\dasharrow\mathbb{P}^2\times\mathbb{P}^2\times\mathbb{P}^2$ given by
$$
[x:y:z:t:w]\mapsto\big([z:t:w],[y:t:w],[x:t:w]\big).
$$
Then $\zeta\circ\theta$ is a morphism that contracts $\mathcal{C}$ to the point $([1:0:0],[1:0:0],[1:0:0])$,
and $-K_Y$ is rationally equivalent to the pullback of the divisor of degree $(1,1,1)$  via~$\zeta\circ\theta$.
Moreover, the image of the 3-fold $Y$ under $\zeta\circ\theta$ is contained in the locus defined by \eqref{equation:3-13-non-toric}.
Using symbolic computer packages, it can be checked that \eqref{equation:3-13-non-toric} defines an integral $3$-fold,
with singular locus $([1:0:0],[1:0:0],[1:0:0])$, which is an ordinary double point.
We see that $\zeta\circ\theta(Y)$ is given by \eqref{equation:3-13-non-toric}, and $X$ can be identified with $\zeta\circ\theta(Y)$.

We have $-K_Y\sim\pi^*(-K_X)$, so  $-K_X^3=-K_Y^3=30$.
Using \cite[Lemma~5.16]{CPS2019}, we get
$$
\mathrm{Aut}^0(X)\cong\mathrm{Aut}^0(Y)\cong\mathrm{Aut}^0(V)\cong\mathbb{C}^\ast,
$$
so $X$ is not toric.
Since $\mathrm{rk}\,\mathrm{Pic}(X)=3=\mathrm{rk}\,\mathrm{Cl}(X)-1$,
$X$ is maximally non-factorial~\cite{CKMS}.

Now, using \cite{JahnkeRadloff2011} and the classification of smooth Fano 3-folds,
we see that $X$ admits a smoothing to a smooth Fano $3$-fold in the deformation family \textnumero 3.13.
One such smoothing can be constructed explicitly as follows. Following Eisenbud--Buchsbaum theory \cite{BuchsbaumEisenbud},
let~$X_{a,b}$ be the codimension three subscheme in $\mathbb{P}^2\times\mathbb{P}^2\times\mathbb{P}^2$
that is given by vanishing of all $4\times 4$ diagonal Pfaffians of the skew-symmetric matrix
$$
\begin{pmatrix}
0 & -2x_1y_1-x_2y_3-x_3y_2 & 2x_1z_1-x_2z_3-x_3z_2 & bx_3 & x_2\\
2x_1y_1+x_2y_3+x_3y_2 & 0 & -2y_1z_1-y_2z_3-y_3z_2 & by_3& y_2\\
-2x_1z_1+x_2z_3+x_3z_2 & 2y_1z_1+y_2z_3+y_3z_2 & 0 & bz_3 & z_2\\
-bx_3 & -by_3 & -bz_3 & 0 & a\\
-x_2 & -x_2 & -z_2 & -a & 0\\
\end{pmatrix}.
$$
In particular, these $5$ equations are the vanishing of $\mathtt{Pf}_1,\dots,\mathtt{Pf}_5$, where

\begin{equation}
\label{equation:3-13-smoothing}
\left\{\aligned
&\mathtt{Pf}_1=(b-a)x_2y_3-(b+a)x_3y_2-2ax_1y_1,\\
&\mathtt{Pf}_2=(b-a)y_2z_3-(b+a)y_3z_2-2ay_1z_1,\\
&\mathtt{Pf}_3=(b-a)x_2z_3-(b+a)x_3z_2+2ax_1z_1,\\
&\mathtt{Pf}_4=bx_1y_1z_3+bx_1y_3z_1+bx_3y_1z_1+bx_3y_2z_3,\\
&\mathtt{Pf}_5=x_1y_1z_2+x_1y_2z_1+x_2y_1z_1+x_2y_3z_2,
\endaligned
\right.
\end{equation}
where $([x_0:x_1:x_2],[y_0:y_1:y_2],[z_0:z_1:z_2])$ are
coordinates on $\mathbb{P}^2\times\mathbb{P}^2\times\mathbb{P}^2$.

There are two interesting relations in the ideal generated by these Pfaffian equations:
\begin{align*}
2a\mathtt{Pf}_4&=-b(z_2\mathtt{Pf}_1-y_3\mathtt{Pf}_3+z_3\mathtt{Pf}_1),\\
2a\mathtt{Pf}_5&=-(z_2\mathtt{Pf}_1-y_2\mathtt{Pf}_3+x_2\mathtt{Pf}_2).
\end{align*}
Hence, as long as $a\neq 0$ we obtain a complete intersection ideal generated by $\mathtt{Pf}_1,\mathtt{Pf}_2,\mathtt{Pf}_3$.
If $[a:b]\ne[1:\pm 1]$, then $X_{a,b}$ is an integral $3$-fold.
Further, if $[1:\pm 1]\ne[a:b]\ne[0:1]$, then  $X_{a,b}$ is a smooth complete intersection,
which is a  Fano 3-folds in the family \textnumero 3.13.
But $X_{0,1}\cong X$ is not a complete intersection,
and its smoothing is clear from \eqref{equation:3-13-smoothing}.

In the remaining part of this section, we will prove that the 3-fold $X$ is K-polystable.
The proof of this result easily follows from \cite[\S 5.23]{Book}.

\paragraph*{\bf Description of the automorphism group}
First, we set
$$
G=\Big\{g\in\mathrm{Aut}(Q)\ \big\vert\ g(C)=C\ \mathrm{and} \ g\big(\{P_1,P_2,P_3\}\big)=\{P_1,P_2,P_3\}\Big\}.
$$
Then $G\cong\mathfrak{S}_3 \times (\mathbb{C}^\ast\rtimes\mumu_2)$ and $\mathfrak{S}_3 \times (\mathbb{C}^\ast\rtimes\mumu_2)$ acts on the quadric $Q$ as follows:
\begin{itemize}
\item if $\sigma\in \mathfrak{S}_3$, then $\sigma$ permutes coordinates $x$, $y$, $z$,

\item if $\lambda\in\mathbb{C}^\ast$, then $\lambda$ acts by $[x:y:z:t:w]\mapsto[\lambda x:\lambda y:\lambda z:\lambda^2 t:w]$,

\item if $\iota\in\mumu_2$, then $\iota$ acts by $[x:y:z:t:w]\mapsto[x:y:z:w:t]$.
\end{itemize}
This action lifts to $Y$. Therefore, we can identify $G$ with a subgroup of the group $\mathrm{Aut}(Y)$.
One can also show that $\mathrm{Aut}(Y)=G$. Since $\mathrm{Aut}(X)\cong\mathrm{Aut}(Y)$,
we can identify $G$ with a subgroup of $\mathrm{Aut}(X)$.

\paragraph*{\bf Description of the $G$-invariant loci}
The conic $C$ is $G$-invariant. Let
$Z=\{x-z=y-z=xy+yz+xz+tw=0\}\subset\mathbb{P}^4$.
Then $Z$ is another smooth $G$-invariant conic that is contained in $Q$. Note that $C\cap Z=\varnothing$.

\begin{lemma}[{\cite[Lemma 5.117]{Book}}]
\label{lemma:5-1-points-curves}
The quadric $Q$ does not contain $G$-invariant points,
and the only irreducible $G$-invariant irreducible curves in $Q$ are the conics $C$ and $Z$.
\end{lemma}

Let $\phi_C\colon Y_C\to Q$ and $\phi_Z\colon Y_Z\to Q$ be the blowup of the conics $C$ and $Z$, respectively.
Let $F_C$ and $F_Z$ be the exceptional surfaces of the blowups $\phi_C$ and $\phi_Z$, respectively.
Then the action of the group $G$ on the quadric $Q$ lifts to actions on $Y_C$ and $Y_Z$, with surfaces $F_C$ and $F_Z$ being $G$-invariant prime divisors over $Q$.

\begin{lemma}[{\cite[Lemma 5.118]{Book}}]
\label{lemma:5-1-blowup-Q-at-C-and-Z}
Let $\mathbf{F}$ be a $G$-invariant prime divisor over $Q$ such that $C_Q(\mathbf{F})$ is not a surface.
Then either $C_Q(\mathbf{F})=C$ and $C_{Y_C}(\mathbf{F})=F_C$,
or $C_Q(\mathbf{F})=Z$ and $C_{Y_Z}(\mathbf{F})=F_Z$.
\end{lemma}

We have the following diagram
$$
\xymatrix@R=4mm{
&&&&V\ar@{->}[lld]_{\eta}\ar@{->}[rrd]^{\vartheta}&&&&\\
X&&Y\ar@{->}[rrd]_{\theta}\ar@{->}[ll]_{\pi}&&&&Y_C\ar@{->}[dll]^{\phi_C}\ar@{->}[rr]^{\xi}&&\mathbb{P}^1\\%
&&&&Q&&&&}
$$
where $\vartheta$ is the blowup of the fibres of the projection $E_C\to C$ over the points $P_1$, $P_2$, $P_3$,
and $\xi$ is a quadric bundle that is given by the linear system $|\phi_C^*(\mathcal{O}_{Q}(1))-E_C|$.

\begin{lemma}
\label{lemma:5-1-beta-E}
One has $\beta(E)=\frac{1}{10}$.
\end{lemma}

\begin{proof}
Let $E_1$, $E_2$, $E_3$ be $\vartheta$-exceptional divisors that are mapped to $P_1$, $P_2$, $P_3$, respectively.
Set $H=(\theta\circ\eta)^*(\mathcal{O}_{Q}(1))$.
Then $\xi\circ\vartheta$ is given by the linear system $|H-E_1+E_2+E_3-E|$.
Let $S$ be a general surface in $|H-E_1+E_2+E_3-E|$. Take $u\in\mathbb{R}_{\geqslant 0}$. Then
$$
(\pi\circ\eta)^*(-K_X)-uE\sim_{\mathbb{R}}3H-2(E_1+E_2+E_3)-uE\sim_{\mathbb{R}} 3S+E_1+E_2+E_3+(3-u)E,
$$
which implies that $(\pi\circ\eta)^*(-K_X)-uE$ is pseudoeffective if and only if $u\leqslant 3$.
Moreover, the divisor  $(\pi\circ\eta)^*(-K_X)-uE$ is nef for $u\leqslant 2$.
Furthermore, if $u\in[2,3]$, the Zariski decomposition of the divisor $(\pi\circ\eta)^*(-K_X)-uE$  is
$$
(\pi\circ\eta)^*(-K_X)-uE\sim_{\mathbb{R}}\underbrace{3H-u(E+E_1+E_2+E_3)}_{\text{positive part}}+\underbrace{(u-2)\big(E_1+E_2+E_3\big)}_{\text{negative part}}.
$$
Now, we can compute $\beta(E)$. Note that
$H^3=2$, $H\cdot E^2=-2$, $E^3=2$, $E_1\cdot E^2=E_2\cdot E^2=E_3\cdot E^2=-1$, $E_1^3=E_2^3=E_3^3=1$,
and other triple intersections of the divisor $H,E,E_1,E_2,E_3$ are zero.
Now, by integrating we get $\beta(E)=\frac{1}{10}$.
\end{proof}

\begin{lemma}
\label{lemma:5-1-beta}
Let $\mathbf{F}$ be a $G$-invariant prime divisor in $V$. Then $\beta(\mathbf{F})\geqslant\frac{1}{20}$.
\end{lemma}

\begin{proof}
By Lemma \ref{lemma:5-1-beta-E}, we may assume that $\mathbf{F}\ne E$.
Let us use notations introduced in the proof of Lemma \ref{lemma:5-1-beta-E}. Then
$\mathbf{F}\sim aS+bE+c(E_1+E_2+E_3)$ for some non-negative integers $a$, $b$, $c$ such that $a>1$.
Therefore, we have $\beta(\mathbf{F})\geqslant \beta(S)$.

Let us compute $\beta(S)$. Take $u\in\mathbb{R}_{\geqslant 0}$. Then
$$
(\pi\circ\eta)^*(-K_X)-uS\sim_{\mathbb{R}}(3-u)H+(u-2)E+uF\sim_{\mathbb{R}}(3-u)S+3E+E_1+E_2+E_3,
$$
so that $(\pi\circ\eta)^*(-K_X)-uS$ is pseudoeffective if and only if $u\leqslant 3$.
If $u\in[0,2]$, then $uE$ is the negative part of the Zariski decomposition of this divisor,
which implies that that
$$
\mathrm{vol}\big((\pi\circ\eta)^*(-K_X)-uS\big)=\mathrm{vol}\big((\pi\circ\eta)^*(-K_X)-uS-uE\big)=u^3-18u+30.
$$
Similarly, if $u\in[2,3]$, the Zariski decomposition of the divisor $(\pi\circ\eta)^*(-K_X)-uS$  is
$$
(\pi\circ\eta)^*(-K_X)-uS
\sim_{\mathbb{R}}\underbrace{(3-u)H}_{\text{positive part}}+\underbrace{uE+(u-2)(E_1+E_2+E_3)}_{\text{negative part}},
$$
which gives $\mathrm{vol}((\pi\circ\eta)^*(-K_X)-uS)=2(3-u)^3$.
By integration we obtain $\beta(S)=\frac{1}{20}$.
\end{proof}

\begin{lemma}
\label{lemma:5-1-beta-F-Z}
Let $\mathbf{F}$ be a $G$-invariant prime divisor over $X$.
Then $\beta(\mathbf{F})\geqslant\frac{1}{20}$.
\end{lemma}

\begin{proof}
Let $\psi\colon \widetilde{Y}\to Y$ be the blowup of the strict transform on $Y$ of the conic $Z$.
By Lemmas \ref{lemma:5-1-blowup-Q-at-C-and-Z} and \ref{lemma:5-1-beta},
we may assume that $\mathbf{F}$ is the $\psi$-exceptional surface.

Let $\widetilde{H}=(\theta\circ\psi)^*(\mathcal{O}_{Q}(1))$,
let $\widetilde{E}_1$, $\widetilde{E}_2$, $\widetilde{E}_3$ be strict transforms
on $\widetilde{Y}$ of the $\theta$-exceptional divisors that are mapped $P_1$, $P_2$, $P_3$, respectively.
Then
$$
(\pi\circ\psi)^*(-K_X)-u\mathbf{F}\sim_{\mathbb{R}}3\widetilde{H}-2\big(\widetilde{E}_1+\widetilde{E}_2+\widetilde{E}_3\big)-u\mathbf{F},
$$
and $(\pi\circ\psi)^*(-K_X)-u\mathbf{F}$ is nef for $u\in[0,1]$.
Thus, if $(\pi\circ\psi)^*(-K_X)-u\mathbf{F}$ is pseudoeffective, then
$$
24-8u=\big((\pi\circ\psi)^*(-K_X)-u\mathbf{F}\big)\cdot\big((\pi\circ\psi)^*(-K_X)-\mathbf{F}\big)^2\geqslant 0.
$$
Hence, the divisor $(\pi\circ\psi)^*(-K_X)-u\mathbf{F}$ is not pseudoeffective for $u>3$.
Then
\begin{multline*}
S_X\big(\mathbf{F}\big)\leqslant\frac{1}{30}\int\limits_0^1\big(3\widetilde{H}-2\big(\widetilde{E}_1+\widetilde{E}_2+\widetilde{E}_3\big)-u\mathbf{F}\big)^3du+\frac{1}{30}\int\limits_{1}^3\mathrm{vol}\big((\pi\circ\psi)^*(-K_X)-\mathbf{F}\big)du=\\
\quad\quad=\frac{1}{30}\int\limits_0^14u^3-18u^2+30du+\frac{2\big(3\widetilde{H}-2\big(\widetilde{E}_1+\widetilde{E}_2+\widetilde{E}_3\big)-\mathbf{F}\big)^3}{30}=\frac{5}{6}+\frac{16}{15}=\frac{19}{10},\quad
\end{multline*}
which implies that $\beta(\mathbf{F})\geqslant\frac{1}{20}$.
\end{proof}

Now, using \cite[Corollary~4.14]{Zhuang}, we see that $X$ is K-polystable.

\end{document}